\documentclass[11pt,reqno]{amsart}
\usepackage{amsfonts, amssymb, amsthm, amsmath, bbm,  amscd, mathtools, multicol, url,  hyperref} %
\usepackage{mathdots} %iddots
\usepackage{arydshln} % dashed lines in matrices
\usepackage{scalerel} %\scaleobj
\usepackage[shortlabels]{enumitem}
\usepackage[all,cmtip]{xy}
\usepackage[dvipsnames]{xcolor}
\usepackage{blkarray} % blockarray
\usepackage{youngtab, ytableau}

\usepackage{tikz}
\usetikzlibrary{cd,arrows}

\usepgflibrary{shapes.geometric}

\tikzstyle{V}=[draw, fill =black, circle, inner sep=0pt, minimum size=1.5pt]
\tikzstyle{C}=[draw, fill =white, circle, inner sep=0pt, minimum size=1.5pt]
\tikzstyle{over}=[draw=white,double=black,line width=2pt, double distance=.5pt]
\numberwithin{equation}{section}
\usepackage[margin = 0.9in]{geometry}
\theoremstyle{definition}

\newtheorem{theorem}{Theorem}[section]

\newtheorem{defn}[theorem]{Definition}
\newtheorem{thm}[theorem]{Theorem}
\newtheorem{lemma}[theorem]{Lemma} 
\newtheorem{prop}[theorem]{Proposition} 
\newtheorem{proposition}[theorem]{Proposition}

\newtheorem{ex}[theorem]{Example}

\newtheorem{rmk}[theorem]{Remark}

\def\<{\langle}
\def\>{\rangle}

\newcommand{\cB}{\mathcal{B}}
\newcommand{\CC}{\mathbb{C}}

\newcommand{\cL}{\mathcal{L}}

\newcommand{\cN}{\mathcal{N}}

\newcommand{\CP}{\mathbb{P}_\mathbb{C}}

\newcommand{\cS}{\mathcal{S}}

\newcommand{\tforall}{~\textup{ for all }~}
 
\newcommand{\fsl}{\mathfrak{sl}} 
 
\newcommand{\fT}{\mathfrak{T}}

\newcommand{\GL}{\mathrm{GL}}

\newcommand{\inv}{^{-1}}

\newcommand{\ld}{\lambda}

\newcommand{\otw}{\textup{otherwise}}

\newcommand{\rk}{\mathop{\textup{rank}}}

\renewcommand{\SS}{\mathbb{S}}

\newcommand{\tA}{\widetilde{A}}
\newcommand{\tB}{\widetilde{B}}
\newcommand{\tcN}{\widetilde{\cN}}

\newcommand{\tcS}{\widetilde{\cS}}

\newcommand{\td}{\widetilde{d}}
\newcommand{\tD}{\widetilde{D}}
\newcommand{\tDe}{\widetilde{\Delta}}

\newcommand{\tGa}{\widetilde{\Gamma}}

\newcommand{\tif}{\textup{if }}
\newcommand{\tphi}{\widetilde{\varphi}}

\newcommand{\TT}{\mathbb{T}}
\newcommand{\tv}{\widetilde{v}}
\newcommand{\tV}{\widetilde{V}}

\newcommand{\ZZ}{\mathbb{Z}}
\renewcommand{\Im}{\mathop{\text{Im}}}
 
\allowdisplaybreaks 
\newcommand{\cupa}{\begin{tikzpicture}[baseline={(0,-.3)}, scale = 0.8]
\draw(-.25,0) -- (1.25,0) -- (1.25,-.9) -- (-.25,-.9) -- cycle;
\begin{footnotesize}
\node at (0,.2) {$1$};
\node at (.5,.2) {$2$};
\node at (1,.2) {$3$};
\end{footnotesize}
\draw[thick] (0,0) .. controls +(0,-.5) and +(0,-.5) .. +(.5,0);

\draw[thick] (1,0) -- +(0,-.9);
\end{tikzpicture}
}

\newcommand{\cupb}{\begin{tikzpicture}[baseline={(0,-.3)}, scale = 0.8]
\draw(-.25,0) -- (1.25,0) -- (1.25,-.9) -- (-.25,-.9) -- cycle;
\begin{footnotesize}
\node at (0,.2) {$1$};
\node at (.5,.2) {$2$};
\node at (1,.2) {$3$};
\end{footnotesize}
\draw[thick] (0.5,0) .. controls +(0,-.5) and +(0,-.5) .. +(.5,0);

\draw[thick] (0,0) -- +(0,-.9);
\end{tikzpicture}
}

\newenvironment{eq}{\begin{equation}}{\end{equation}}

\title{A study of irreducible components of Springer fibers using quiver varieties}
\author[M.S. Im, C.-J. Lai, and A. Wilbert]{Mee Seong Im,  Chun-Ju Lai, and Arik Wilbert}
\address{Department of Mathematical Sciences, United States Military Academy, West Point, NY 10996, USA}
\curraddr{Department of Mathematics, United States Naval Academy, Annapolis, MD 21402, USA}
    \email{meeseongim@gmail.com (Im)}
\address{Institute of Mathematics, Academia Sinica, Taipei 10617 Taiwan}
    \email{cjlai@gate.sinica.edu.tw (Lai)}
\address{Department of Mathematics, University of Georgia, Athens, GA 30602, USA}
\email{arik.wilbert@uga.edu (Wilbert)}

\begin{document}

\begin{abstract}
It is a remarkable theorem by Maffei--Nakajima that the Slodowy variety, which is a subvariety of the resolution of the nilpotent cone, can be realized as a Nakajima quiver variety of type A.
However, the isomorphism is rather implicit as it takes to solve a system of equations in which the variables are linear maps.
In this paper, we construct solutions to this system under certain assumptions.
This establishes an explicit and efficient way to compute the image of a complete flag contained in the Slodowy variety under the Maffei--Nakajima isomorphism and describe these flags in terms of quiver representations. 
As Slodowy varieties contain Springer fibers naturally,
we can use these results to provide an explicit description of the irreducible components of two-row Springer fibers in terms of a family of kernel relations via quiver representations.
%last update: \currenttime, \today
\end{abstract}
 
\maketitle 

%=============================================
\section{Introduction}

\subsection{An overview} \label{sec:background}
Quiver varieties were used by 
Nakajima in~\cite{Na94, Nak98} to provide a geometric construction of the universal enveloping algebra for symmetrizable Kac-Moody Lie algebras altogether with their integrable highest weight modules.  
It was shown by Nakajima that the cotangent bundle of partial flag varieties can be realized as a quiver variety.
He also conjectured that the Slodowy varieties, i.e., resolutions of slices to the adjoint orbits in the nilpotent cone, can be realized as quiver varieties.
In type A, this conjecture was proved by Maffei,~\cite[Theorem~8]{Maf05}, 
thereby establishing a precise connection between quiver varieties and Slodowy varieties. 
Precisely speaking, a Slodowy variety $\tcS_\ld$ parametrized by a partition $\ld$ is isomorphic to a quiver variety $M(d,v)$ of type A under a restriction (see \eqref{eq:dv}) on the dimention vectors $d$ and $v$. The relation between the partition and the dimension vectors can be found in \eqref{def:ldv}. 

However, this Maffei--Nakajima isomorphism is in general implicit in the sense that it requires to solve a system of equations in which variables are linear maps.
Our main result is to provide an explicit solution to the system (cf. Lemma~\ref{lem:hell} when $\ld$ is a two-row partition, and Lemma~\ref{lem:hell2} for an arbitrary partition) under assumption \eqref{eq:T}, and hence obtain an explicit flag realization of type A quiver varieties that arises from a partition (cf. Theorem~\ref{thm:Fi} for the two-row case, and Theorem~\ref{thm:Fi2} for an arbitrary partition).

Aside from quiver varieties, an important geometric object in our article is the Springer fiber, which plays a crucial role in the geometric construction of representations of Weyl groups (cf. \cite{Spr76, Spr78}). 
In general, Springer fibers are singular and decompose into many irreducible components.
Characterization of irreducible components has been a challenging problem. 
Little is known when the corresponding partition of Springer fiber has more than two rows.
As Springer fibers are naturally contained in the Slodowy varieties, it makes sense to describe the image of a Springer fiber in the corresponding quiver variety, and it is well-know that this image is Lusztig's Lagrangian subvariety (see \cite{Lu91, Lu98}). %Na94, Nak98, }). 
Moreover, irreducible components of these Lagrangian subvarieties actually form a crystal base structure (see \cite{Sai02, KS19}). Our results however make preceding constructions explicit. 

When the (type A) Springer fibers correspond to nilpotent endomorphisms having exactly two blocks (i.e., the two-row case),
a characterization of its irreducible components in terms of flags and so-called cup diagrams (cf. Section~\ref{sec:cupdiag}) was obtained by Stroppel--Webster, \cite{SW12}, based on an earlier work by Fung, \cite{Fun03}. 
Our work allows for a translation (see Theorem~\ref{thm:main1}) of the results of Stroppel--Webster to the quiver variety, i.e.,  the new relations on the quiver representation side that correspond to the cup/ray relations in \cite[Proposition~7]{SW12} as below:
\eq
\begin{split}
\textup{Cup relation for } \overset{{}_i\hspace{.9mm}{}_j}{\cup} 
&
\quad
%  =
\Leftrightarrow 
\quad
F_j = 
x^{-\frac{1}{2}(j-i+1)} F_{i-1}
\\
&
\quad
\Leftrightarrow 
\quad
\ker B_{i-1} B_{i} \cdots B_{\frac{j+i-3}{2}}=\ker A_{j-1} A_{j-2} \cdots A_{\frac{j+i-1}{2}},
\end{split}
\endeq
\eq
\begin{split}
\textup{Ray relation for } \overset{{}_i}{|} 
&
\quad
%  = 
\Leftrightarrow 
\quad
F_i = x^{-\frac{1}{2}(i-\rho(i))}(x^{n-k-\rho(i)} F_n ) 
\\
&
\quad
\Leftrightarrow 
\quad
\begin{cases}
\:\: B_iA_i = 0
&\tif c(i) \geq 1, 
\\
B_{i} B_{i+1} \cdots B_{n-k-1}\Gamma_{n-k} = 0
&\tif c(i) = 0,
\end{cases}
\end{split}
\endeq
where $\rho(i) \in \ZZ_{>0}$ counts the number of rays (including itself) to the left of $i$ , and $c(i) = \frac{i-\rho(i)}{2}$ is the total number of cups to the left of $i$. 

An application of the above characterization of type A is to obtain a characterization of the irreducible components of two-row Springer fibers of classical type by applying a recent work by Henderson--Licata and Li (cf. \cite{HL14, Li19}), in which two-row Springer fibers of classical type are realized by fixed-point subvarities of type A quiver varieties under certain automorphisms that arise from folding framed quivers.
Although the characterization is discovered by a tedious computation using quiver representations (see an unpublished manuscript \cite[\S5--7]{ILW19}), we realized that a more concise proof using flags (without quiver representations) is possible. 
Therefore, we will publish the result in an accompanying paper (\cite{ILW20b}).

It would be interesting to investigate if the relations on the quiver variety for the irreducible components generalize to nilpotent endomorphisms having more than two Jordan blocks. For these endomorphisms, characterizing the irreducible components remains an open problem.     
 
\subsection{Organization}
In Section~\ref{sec:Naka} we recall the construction of Nakajima quiver varieties.
Following Maffei, we will focus on the realization viewing geometric points as orbits of quiver representations rather than the original construction via geometric invariant theory (GIT).

In Section~\ref{sec:iso} we describe the Maffei--Nakajima isomorphism and the aforementioned system of equations. Details are provided mainly for the two-row case for readability.

Section~\ref{sec:exp} is devoted to constructing an explicit solution to the system, and thus establishing an efficient and explicit way to obtain the corresponding flag in the (not necessarily two-row) Slodowy variety from a quiver representation.

Finally, in Section~\ref{sec:appl} we describe irreducible components of Lusztig's Lagragian subvarities in type A quiver varieties in the two-row case. We also remark the connection to the irreducible components of two-row Springer fibers of classical type.

\subsection*{Acknowledgment}
M.S.I. would like to thank the University of Georgia for organizing the Southeast Lie Theory Workshop X, where our collaborative research group initially came together. 
The authors also thank the AGANT group at the University of Georgia for supporting this project. 
This project was initiated as a part of the Summer Collaborators Program 2019 at the School of Mathematics at the Institute for Advanced Study (IAS). The authors thank the IAS for providing an excellent working environment.
We are grateful for useful discussions with Jieru Zhu during the early stages of this project. 
We thank Anthony Henderson, Yiqiang Li, and Hiraku Nakajima for helpful clarifications, Vasily Krylov for pointing out a typo, and Catharina Stroppel for helpful comments on earlier drafts of our manuscript. 
A.W. was  partially supported by the ARC Discover Grant DP160104912 ``Subtle Symmetries and the Refined Monster''.
C.-J. L. is partially supported by the MoST grant 109-2115-M-001-011-MY3, 2020 -- 2023.

%=============================================
\section{Nakajima quiver varieties of type A}\label{sec:Naka}
%=============================================
We begin with the traditional approach to Nakajima quiver varieties via geometric invariant theory (GIT) quotients. 
In order to make computation easier, we then adopt an equivalent realization in which each geometric point is an orbit of a certain quiver representation space, following \cite{Maf05}. 

%=============================================
\subsection{Quiver representation space $R(d,v)$}
%=============================================
Fix a Dynkin quiver of type $A_{n-1}$. 
Its {\em framed quiver} $Q'$ is obtained by adding a shadow vertex that connects to each vertex.
We consider its {\em framed double quiver} $Q$ obtained from $Q'$ by replacing each arrow with two sided ones.
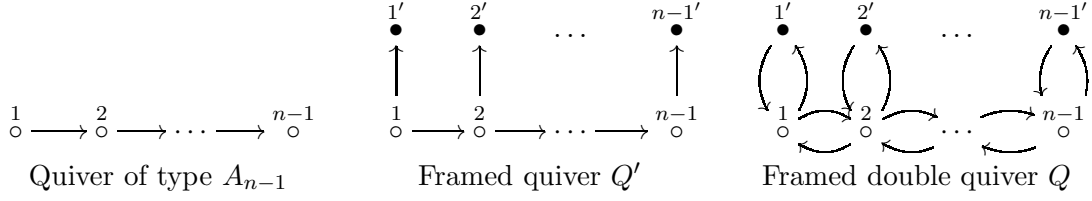
\[
\begin{tabular}{ccc}
$\xymatrix@-.5pc{
{~}\phantom{\stackrel{1'}{\bullet}}
\\ 
\stackrel{1}{\circ} \ar[r] 
& \stackrel{2}{\circ}\ar[r] 
& \cdots \ar[r] 
& \stackrel{n-1}{\circ}   
\\ 
}
$ 
&
\quad
$\xymatrix@-.5pc{
\stackrel{1'}{\bullet}
& \stackrel{2'}{\bullet}
& \ldots 
&\stackrel{n-1'}{\bullet}
\\ 
\stackrel{1}{\circ} \ar[r] \ar[u] 
& \stackrel{2}{\circ}\ar[r] \ar[u]
& \cdots \ar[r] 
& \stackrel{n-1}{\circ} \ar[u]  \\ 
}
$ 
&
\quad
$\xymatrix@-.5pc{
\stackrel{1'}{\bullet} \ar@/_.75pc/[d]
& \stackrel{2'}{\bullet} \ar@/_.75pc/[d]
& \ldots 
&\stackrel{n-1'}{\bullet} \ar@/_.75pc/[d]
\\ 
\stackrel{1}{\circ}  \ar@/^.65pc/[r]_{} \ar@/_.75pc/[u] 
& \stackrel{2}{\circ}\ar@/^.65pc/[l] \ar@/^.65pc/[r]_{} \ar@/_.75pc/[u]
& \cdots \ar@/^.65pc/[l] \ar@/^.65pc/[r]_{} 
& \stackrel{n-1}{\circ} \ar@/^.65pc/[l] \ar@/_.75pc/[u]  \\ 
}$ 
\\[1.5cm]
Quiver of type $A_{n-1}$
&
Framed quiver $Q'$
&
Framed double quiver $Q$
\end{tabular}
\]
\defn
A {\em quiver representation} of $Q$ is obtained by replacing vertices $i$ (and $i'$, resp.) with vector spaces $V_i$ (and $D_i$, resp.) as well as replacing arrows with the following linear maps, for $1\leq i \leq n-1, 1\leq j \leq n-2$:
\eq\label{eq:ABGD}
A_j:V_j\to V_{j+1},
\quad
B_j:V_{j+1}\to V_j, 
\quad
\Gamma_i:D_i\to V_i, 
\quad
\Delta_i:V_i\to D_i.
\endeq
Denote by $d = (\dim D_i)_i$ and $v = (\dim V_i)_i$ the corresponding dimension vectors,
and let $R(d,v)$ be the space of quiver representations of $Q$ with dimension vectors $d$ and $v$. 
Now we identify $R(d,v)$ as the space of quadruples $(A,B,\Gamma, \Delta)$ of collections of linear maps of the form in \eqref{eq:ABGD}.
\enddefn
\rmk\label{rmk:0}
In the rest of article, any space with an ineligible subscript is understood as a zero space (e.g., $V_0 = \{0\}, V_{n} =\{0\}$). Any linear map with an ineligible subscript is understood as a zero map (e.g., $A_{n-1}:V_{n-1}\to \{0\}$, $B_0:V_1 \to \{0\}$).
\endrmk

We will focus on those dimension vectors such that the corresponding quiver variety (yet to define) affords a flag realization. 
Nakajima provided such a flag realization for the special case (see Proposition~\ref{prop:MS} for details)
\eq\label{eq:Nakdv}
d = (n,0,\ldots,0),
\quad
v= (n-1, n-2, \ldots, 1),
\endeq 
which is then generalized by Maffei (cf. \cite[\S 1.4]{Maf05}) for $(d,v)$ satisfying the conditions below, for $1\leq i\leq n-1$:
\eq\label{eq:dv}
v_i = (n-i) - (n-i-1) d_{n-1} -  (n-i-2) d_{n-2} - \ldots - d_{i+1}.
\endeq
From now on, the assumption \eqref{eq:dv} applies to the remaining part of this article.
Note that \eqref{eq:Nakdv} is recovered by setting $d_1= n, d_2 = \ldots = d_{n-1} = 0$ in \eqref{eq:dv}.

Given a dimension vector $d = (d_1, \ldots, d_{n-1})$, we assign a partition $\ld = \ld(d)$  such that this partition is of type  $((n-1)^{d_{n-1}}, \ldots, 1^{d_1})$. That is, for $i < n$, $d_i$ counts the number of appearance of $i$ in $\ld$, or
\eq\label{def:ldv}
\ld = \ld(d) =  (\ld_1, \ld_2, \ldots, \ld_r),
\quad
d_i = \#\{j \in \{1, \ldots, r\} ~|~ \ld_j = i \}.  
\endeq
Note that the vector $d$ in \eqref{eq:Nakdv} produces the partition $(1,1,\ldots,1)$ of type $(1^n)$. 
%while any one-row partition is never produced from the rule in \eqref{def:ldv}.
Write $R_{\ld} = R(d,v)$ for $d, v$ satisfying \eqref{eq:dv}.

We begin with the case where the partition $\ld$ has exactly two rows, i.e.,  either $d_k = d_{n-k} = 1$ for some $k < n-k$, or $d_k = 2$ if $2k = n$; while all the other $d_i$'s are zeroes. 
The dimension vector $v$ is then determined by $d$ via \eqref{eq:dv}, and hence we have
\eq\label{eq:2rowdv}
d_i = \delta_{i,k} + \delta_{i, n-k},
\quad
v_i = \begin{cases}
\quad i &\tif i \leq k, \\
\quad k &\tif k\leq i \leq n-k, \\
n-i &\tif i \geq n-k.
\end{cases}
\endeq
From now on, we fix elements $e$ and $f$ so that
\eq\label{def:ef}
D_{k} = 
\begin{cases}
\<f\> &\tif n-k > k;
\\
\<e,f\> &\tif n-k = k,
\end{cases}
\quad 
D_{n-k} = 
\begin{cases}
\<e\> &\tif n-k > k;
\\
D_k &\tif n-k = k.
\end{cases}
\endeq
A general quiver representation in $R_\ld$ for $n-k < k$ (and for $n-k=k$, resp.) is of the form below in Figure~\ref{figure:Sn-kk} (and Figure~\ref{figure:Skk}, resp.).
%-----------------------------------------------------------------
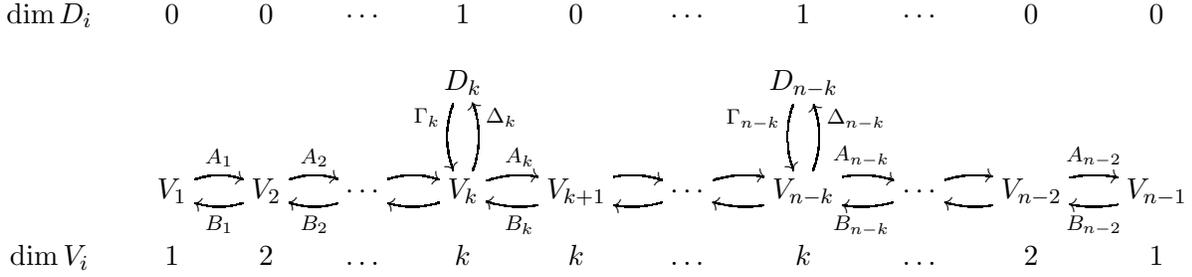
\begin{figure}[ht!]
\caption{Quiver representations in $R_{(n-k,k)}$ and dimension vectors, where $k<n-k$.}
 \label{figure:Sn-kk}
\[ 
\xymatrix@C=18pt@R=9pt{
\dim D_i&0&0&\cdots &1&0&\cdots &1&\cdots &0&0
\\
&
& &  &   
\ar@/_/[dd]_<(0.2){\Gamma_k}
D_k
& &  &  
\ar@/_/[dd]_<(0.2){\Gamma_{n-k}}
D_{n-k}
 & & &&
\\ & & & & & & &  && 
\\ 
&%\underset{=\mathbb{C}}
{V_1}  
	\ar@/^/[r]^{A_1} 
	& 
	\ar@/^/[l]^{B_1}
%\underset{=\mathbb{C}^2}
{V_2}
	\ar@/^/[r]^{A_2}
	&
	\ar@/^/[l]^{B_2}
	\cdots   
%\underset{=\mathbb{C}^{k-1}}{V_{k-1}} 
	\ar@/^/[r]%^{A_{k-1}} 
	& 
	\ar@/^/[l]%^{B_{k-1}} 
%\underset{=\mathbb{C}^k}
{V_k}  
	\ar@/_/[uu]_>(0.8){\Delta_k} 
	\ar@/^/[r]^{A_{k}}
	&   
%\underset{=\mathbb{C}^{k}}
{V_{k+1}}  
	\ar@/^/[l]^{B_k}
	\ar@/^/[r]%^{A_{k+1}}
	&
	\ar@/^/[l]%^{B_{k+1}}
	\cdots
%\underset{=\mathbb{C}^k}{V_{n-k-1}} 
	\ar@/^/[r]%^{A_{n-k-1}}
	&  
%\underset{=\mathbb{C}^{k}}
{V_{n-k}} 
	\ar@/_/[uu]_>(0.8){\Delta_{n-k}}
	\ar@/^/[r]^{A_{n-k}}
	\ar@/^/[l]%^{B_{n-k-1}}
	& 
%\underset{=\mathbb{C}^{k-1}}{V_{n-k+1}}  
%	\ar@/^/[l]^{B_{n-k}}
%	\ar@/^/[r]%^{A_{n-k+1}}
%	&
	\ar@/^/[l]^{B_{n-k}}
	\cdots
	\ar@/^/[r]%^{A_{n-3}}
	&
	\ar@/^/[l]%^{B_{n-3}}
%\underset{=\mathbb{C}^2}
{V_{n-2}}  
	\ar@/^/[r]^{A_{n-2}}
	& 
%\underset{=\mathbb{C}^1}
{V_{n-1}} 
\ar@/^/[l]^{B_{n-2}}  
\\
\dim V_i & 1&2&\ldots&k&k&\ldots&k&\ldots&2&1
}
\] 
\end{figure}
%-----------------------------------------------------------------
%-----------------------------------------------------------------
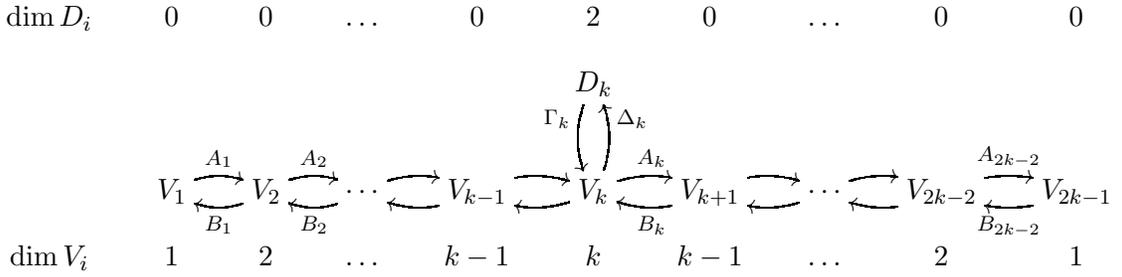
\begin{figure}[ht!]
\caption{Quiver representations in $R_{(k,k)}$ and the dimension vectors.}
 \label{figure:Skk}
\[ 
\xymatrix@C=18pt@R=9pt{
\dim D_i& 0&0&\ldots&0&2&0&\ldots&0&0
\\
&
& &  &   &
\ar@/_/[dd]_<(0.2){\Gamma_k}
D_k
& &  &  &
\\ & & & & & & &  && 
\\ 
&%\underset{=\mathbb{C}}
{V_1}  
	\ar@/^/[r]^{A_1} 
	& 
	\ar@/^/[l]^{B_1}
%\underset{=\mathbb{C}^2}
{V_2}
	\ar@/^/[r]^{A_2}
	&
	\ar@/^/[l]^{B_2}
	\cdots
	\ar@/^/[r]%^{A_{3}} 
	&   
%\underset{=\mathbb{C}^{k-1}}
	\ar@/^/[l]%^{B_{k-1}} 
{V_{k-1}} 
	\ar@/^/[r]%^{A_{k-1}} 
	& 
	\ar@/^/[l]%^{B_{k-1}} 
%\underset{=\mathbb{C}^k}
{V_k}  
	\ar@/_/[uu]_>(0.8){\Delta_k} 
	\ar@/^/[r]^{A_{k}}
	&   
	\ar@/^/[l]^{B_k}
%\underset{=\mathbb{C}^{k}}
{V_{k+1}}  
	\ar@/^/[r]%^{A_{k+1}}
	& 
	\ar@/^/[l]%^{B_{n-k}}
	\cdots
	\ar@/^/[r]%^{A_{n-3}}
	&
	\ar@/^/[l]%^{B_{n-3}}
%\underset{=\mathbb{C}^2}
{V_{2k-2}}  
	\ar@/^/[r]^{A_{2k-2}}
	& 
%\underset{=\mathbb{C}^1}
{V_{2k-1}} 
\ar@/^/[l]^{B_{2k-2}}  
\\
\dim V_i & 1&2&\ldots&k-1&k&k-1&\ldots&2&1
}
\] 
\end{figure}
%-----------------------------------------------------------------
We will focus on the two-row case in the first half of this paper, and deal with the general case where $\ld$ can be arbitrary in Section~\ref{sec:rrows}.
%=============================================
\subsection{Stable locus $\Lambda^+(d,v)$}
%=============================================

Following Nakajima, an element $(A,B,\Gamma,\Delta) \in R(d,v)$ is called {\em admissible} if the Atiyah-Drinfeld-Hitchin-Manin (ADHM) equations are satisfied. Equivalently, for all $1\leq i\leq n-1$,
\eq\label{eq:L1}
B_i A_i =  A_{i-1} B_{i-1} + \Gamma_i \Delta_i.
\endeq
An admissible element is called {\em stable} if, for each collection $U = (U_i \subseteq V_i)_i$ of subspaces satisfying
\eq
\Gamma_i(D_i) \subseteq U_i, 
\quad
A_i(U_i) \subseteq U_{i+1}, 
\quad
B_i(U_{i+1}) \subseteq U_i 
\quad\quad 
\tforall i,
\endeq
it follows that $U_i = V_i$ for all $i$.
We will use the following equivalent notion of stability due to Maffei:
%-----------------------------------------------------------------
\begin{lemma}[{\cite[Lemmas~14, 2]{Maf05}}]
An admissible element $(A,B,\Gamma,\Delta) \in R(d,v)$ is stable if and only if, for all $1\leq i\leq n-1$,
\eq\label{eq:L2}
\Im \: A_{i-1} + \sum_{j\geq i} \Im \: \Gamma_{j\to i} = V_i,
\endeq
where it is understood that $A_0 =0$, and that $\Gamma_{j\to i}$, for all $i, j$, is the natural composition from $D_j$ to $V_i$, i.e.,
\eq\label{def:Gaij}
\Gamma_{j\to i} = 
\begin{cases} 
B_i  \ldots B_{j-1} \Gamma_j &\tif j \geq i;
\\
A_{i-1} \ldots A_j \Gamma_j &\tif j \leq i.
\end{cases}
\endeq
\end{lemma}
%-----------------------------------------------------------------
\defn
The {\em stable locus} $\Lambda^+(d,v)$ is the subspace of $R(d,v)$ consisting of elements that are admissible  and stable, i.e., 
\eq\label{def:L+}
\Lambda^+(d,v) = \{(A,B,\Gamma,\Delta)  \in R(d,v) \mid \eqref{eq:L1}, \eqref{eq:L2} \textup{ hold}\},
\endeq
We denote by $\Lambda_{\ld}$ as the set of admissible representations in $R_{\ld}$, and $\Lambda^+_{\ld}$ as its stable locus.
\enddefn
%-----------------------------------------------------------------
%\ex\label{ex:21}
%Take $n=3, k=1$. The quiver representations in $S_{2,1}$ are described as below:
%\[
% \xymatrix@-1pc{
%D_1=\mathbb{C} \ar@/^/[dd]^{\Gamma_1} & &   D_2=\mathbb{C}  \ar@/^/[dd]^{\Gamma_2}  \\ 
% & & \\ 
%V_1=\CC\ar@/^/[uu]^{\Delta_1} \ar@/^/[rr]^{A_1} & & \ar@/^/[ll]^{B_1}  V_2=\CC  \ar@/^/[uu]^{\Delta_2}  \\ 
% }
%\] 
%We note that the vector spaces here are all of dimension one, and hence the linear maps $A_i, B_i, \Gamma_i, \Delta_i$ can be realized as scalar multiplications by $a, b, \gamma_i, \delta_i \in \CC$, respectively.
%The admissibility conditions \eqref{eq:L1} are then
%\eq\label{eq:L1_21}
%\gamma_1 \delta_1 =ab,
%\quad
%\gamma_2 \delta_2 + ab=0,
%\endeq
%while the stability condition \eqref{eq:L2} is
%\eq\label{eq:L2_21}
%\Im\: \Gamma_1  + \Im \: B_1 \Gamma_2 = \mathbb{C},
%\quad
%\Im \: A_1  + \Im \: \Gamma_2 = \mathbb{C}.
%\endeq
%\endex
%-----------------------------------------------------------------
%=============================================
\subsection{Nakajima quiver variety $M(d,v)$}
%=============================================
Let  $R'(d,v)$ be the quiver representation space of the framed quiver $Q'$ instead of the framed double quiver $Q$  for the same dimension vectors $d$ and $v$.
Following \cite{Na94}, the representation space $R(d,v)$ is naturally a symplectic variety by identifying it with the cotangent bundle of $R'(d,v)$.
Let $V = \prod_i V_i$, $D = \prod_i D_i$.
The reductive group  $\GL(V) = \prod_i \GL(V_i)$ acts on $R'(d,v)$ by symplectic base change automorphisms, and hence induces a moment map
\eq
\mu_{\textup{mom}}: T^*(R'(d,v)) \to \bigoplus_{i=1}^{n-1} \mathfrak{gl}(V_i).
\endeq
For the character
$\chi: \GL(V) \to \CC^*,
g = (g_i)_i \mapsto \prod_i \det g_i$,
define a graded algebra $\mathcal{A}^\chi = \bigoplus_i \mathcal{A}^{\chi}_i$ whose $i$th graded component is given by
\eq
\mathcal{A}^{\chi}_i = \{\textup{regular } f \in \CC[\mu_{\textup{mom}}\inv(0)] ~|~ f(g.x) = \chi^i(g) f(x)
\textup{~for all~}g\in \GL(V), x \in \mu_{\textup{mom}}\inv(0)\}.
\endeq
\defn
The Nakajima quiver variety is the GIT quotient
\eq
M(d,v) = \mu_{\textup{mom}}\inv(0)//_\chi \GL(V) = \textup{Proj}(\mathcal{A}^\chi).
\endeq
Denote also by $M_{\ld}$ the Nakajima quiver variety for $R_{\ld}$.
\enddefn
\rmk
The geometric points of the scheme $M(d,v)$ are in bijection with $\GL(V)$-orbits of $\Lambda^+(d,v)$.
For our purpose, it is more convenient to use the identification 
\eq\label{def:M}
M(d,v) = \Lambda^+(d,v)/\GL(V)
\endeq 
as an orbit space up to the following $\GL(V)$-action on $R(d,v)$, for $g=(g_i)_i \in \GL(V)$:
\eq\label{def:GL action}
\begin{aligned}
&g. (A,B,\Gamma, \Delta) = 
((g_{j+1}A_j g_j\inv)_j, 
\quad(g_j B_j g_{j+1}\inv)_j, 
\quad(g_i\Gamma_i)_i, 
\quad(\Delta_ig_i\inv)_i).%;
%\\
%&h\cdot (A,B,\Gamma, \Delta) = (A,B,(\Gamma_ih_i\inv)_i, (h_i\Delta_i)_i), & h \in \GL(D).
\end{aligned}
\endeq
\endrmk
For $(A,B,\Gamma, \Delta)\in R(d,v)$, we abbreviate its $\GL(V)$-orbit by $[A,B,\Gamma,\Delta]$.
We denote the projection onto $M(d,v)$ by
\eq\label{def:p}
p_{d,v}: \Lambda^+(d,v) \to M(d,v),
\quad
(A,B,\Gamma,\Delta) \mapsto [A,B,\Gamma,\Delta].
\endeq
Denote also by $p_{\ld}$ for the projection onto $M_{\ld}$.

%=============================================
\section{Maffei-Nakajima isomorphism}\label{sec:iso}
%=============================================

%=============================================
\subsection{Springer fibers and Slodowy varieties} 
%=============================================
Fix integers $n, k$ such that $n \geq 1$ and $n-k \geq k \geq 1$.
Let $\mathcal{N}$ be the variety of nilpotent elements in $\mathfrak{gl}_n(\CC)$. Let $G=\GL_n(\CC)$. 
One may parametrize the $G$-orbits in $\mathcal{N}$ using partitions of $n$ by associating a nilpotent endomorphism to the list of the dimensions of the Jordan blocks. 
%We denote the $G$-orbit associated to the partition $\lambda$ by $\mathcal{O}_{\lambda}$. 

Denote the complete flag variety in $\CC^n$ by
\eq
\cB = \{ F_\bullet = (0 = F_0 \subset F_1 \subset \ldots \subset F_n = \CC^n)\mid\dim F_i = i \tforall 1\leq i\leq n\}.
\endeq
Let  
$\mu: \tcN \to \cN, (u, F_\bullet) \mapsto u$ 
be the Springer resolution,
where
\eq
\tcN = T^*\cB \cong \{(u,F_\bullet)\in \cN \times \cB \mid u (F_i) \subseteq F_{i-1} \tforall i \}.
\endeq
\defn
The Springer fiber of $x\in \cN$ is denoted by $\cB_x = \mu\inv(x)$. We adapt the identification
\eq
\cB_x \equiv \{ F_\bullet \in \cB ~|~ x (F_i) \subseteq F_{i-1} \tforall i  \}.
\endeq
\enddefn
%=============================================
%=============================================
For each $x\in \mathcal{N}$, the
Slodowy transversal slice 
of (the $G$-orbit of) 
$x$ is the variety 
\eq
\mathcal{S}_x =\{ u\in \mathcal{N}\mid [u-x,y]=0\}, 
\endeq
where $(x,y,h)$ is an $\fsl_2$-triple in $\cN$ (here $(0,0,0)$ is also considered as an $\mathfrak{sl}_2$-triple).  
%Define $\mathcal{S}_{r,x}=\mathcal{S}_x \cap \overline{\mathcal{O}}_{\lambda_r}$. 

\defn
The Slodowy variety associated to $x \in \cN$ is given by
\eq
\tcS_{x}=\mu^{-1}(\mathcal{S}_{x}) = \{ (u, F_\bullet) \in \cN\times \cB \mid [u-x,y]=0,  \:\:   u (F_i) \subseteq F_{i-1} \textup{ for all }i \}.
\endeq  
In particular, $x \in \cS_x$ and hence $\cB_x = \mu\inv(x) \subset \mu\inv(\cS_x) = \tcS_x$.
\enddefn

For now, we fix $x\in \cN$ to be of Jordan type $\ld =(n-k,k)$, i.e., $\ld$ corresponds to a two-row Young diagram. 
We write  
\eq
\cB_{\ld}:= \cB_{x},
\quad
\cS_{\ld} := \cS_{x},
\quad
\tcS_{\ld} := \tcS_{x},
\endeq
and we also call $\cB_\ld$ a {\em two-row Springer fiber}.

%=============================================
\subsection{Nakajima's isomorphism} 
%=============================================
It is first proved in \cite[Thm.~7.2]{Na94} that there is an explicit isomorphism $M(d,v) \to \tcS_x$ for certain $d,v,x$  using a different stability condition.
Here we recall a variant due to Maffei that suits our need. 
%-------------------------------------------------------------
\begin{prop}[{\cite[Lemma~15]{Maf05}}]\label{prop:MS}
If $(d,v)$ satisfies \eqref{eq:Nakdv}, then the assignment below defines an isomorphism $\tphi(d,v): M(d, v) \to \tcS_x, [A, B, \Gamma, \Delta] \mapsto (x, F_\bullet)$, 
where
\eq
 x = \Delta_1 \Gamma_1,
 \quad
F_\bullet =  (0 \subset \ker\Gamma_1 
\subset \ker \Gamma_{1\to2} 
\subset
\ldots \subset \ker \Gamma_{1\to n}).
\endeq
\end{prop}
%-------------------------------------------------------------
%-------------------------------------------------------------
In other words, Proposition~\ref{prop:MS} provides a flag realization of $M(d,v)$ when the quiver representations are of the form in Figure~\ref{figure:tS}.
Equivalently, it works for the partition $\ld(d) = \bar{1} := (1,1,\ldots, 1)$.
For other $(d,v)$ that satisfy \eqref{eq:dv},
a flag realization of $M(d,v)$ is also available, with the help of an auxiliary quiver representation space.
For our need, we will only focus on the special case $M(d,v) = M_\ld$ for some partition $\ld =(n-k,k)$ (see \eqref{def:ldv}). 
%=============================================
\subsection{Modified quiver representation space $\widetilde{R}_{\ld}$}
%=============================================
Following \cite{Maf05}, we utilize a modified quiver representation space $\widetilde{R}_{\ld}$ as in Figure~\ref{figure:tS} below:

\begin{figure}[ht!]
\caption{Modified quiver representations in $\widetilde{R}_{\ld}$.}
 \label{figure:tS}
\[ 
\xymatrix@C=18pt@R=9pt{
\dim \tD_i& n&0&\ldots&0&0
\\
&\ar@/_/[dd]_<(0.2){\tGa_1}
\tD_1
& &  &  &
\\ & & & & & & &  && 
\\ 
&
{\tV_1}  
	\ar@/_/[uu]_>(0.8){\tDe_1} 
	\ar@/^/[r]^{\tA_{1}}
	&   
	\ar@/^/[l]^{\tB_1}
{\tV_{2}}  
	\ar@/^/[r]
	& 
	\ar@/^/[l]
	\cdots
	\ar@/^/[r]
	&
	\ar@/^/[l]
{\tV_{n-2}}  
	\ar@/^/[r]^{\tA_{n-2}}
	& 
%\underset{=\mathbb{C}^1}
{\tV_{n-1}} 
\ar@/^/[l]^{\tB_{n-2}}  
\\
\dim V_i & n-1&n-2&\ldots&2&1
}
\] 
\end{figure}
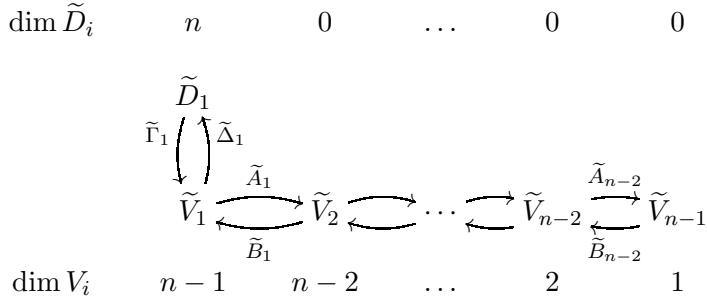

We fix elements $e_1, \ldots, e_{n-k}, f_1, \ldots, f_k$, and we define
 vector spaces $\tD_1, \tV_1, \ldots, \tV_{n-1}$ by 
$\tD_1 = D'_0,
\quad
\tV_i = V_i \oplus D'_i,
$
where
\eq\label{def:D'}
D'_i = \begin{cases}
\< e_1, \ldots, e_{n-k-i}, f_1, \ldots, f_{k-i}\>
&\tif i \leq k-1;
\\
\qquad\hspace{2mm} \< e_1, \ldots, e_{n-k-i}\>
&\tif k \leq i \leq n-k-1;
\\
\qquad\qquad\quad \{0\} &\tif n-k \leq i \leq n-1.
\end{cases}
\endeq
The spaces $D^{(h)}_j$ in \cite{Maf05} are identified as follows:
\eq
\begin{split} 
\<e_i\> \equiv D^{(i)}_{n-k},
\quad
\<f_i\> \equiv D^{(i)}_{k}&\quad\tif n > 2k, 
\\
\<e_i, f_i\> \equiv D^{(i)}_k &\quad\tif n =2k.
\end{split}
\endeq
%
%for the case $n-k > k$,
%\begin{align}
%\tD_1 &= D_k^{(1)} \oplus \ldots \oplus D_k^{(k)} \oplus D_{n-k}^{(1)} \oplus \ldots D_{n-k}^{(n-k)},
%\\
%\tV_i &= V_i \oplus D_k^{(1)} \oplus \ldots \oplus D_k^{(k-i)} \oplus D_{n-k}^{(1)} \oplus \ldots D_{n-k}^{(n-k-i)},
%\end{align}

%For the special case $n-k = k$, it is given by
%\begin{align}
%\tD_1 &= D_k^{(1)} \oplus \ldots \oplus D_k^{(n)},
%\\
%\tV_i &= V_i \oplus D_k^{(1)} \oplus \ldots \oplus D_k^{(k-i)}.
%\end{align}
Denote by $\td = (\dim \tD_i)_i = (n,0, \ldots, 0), \tv = (\dim \tV_i)_i = (n-1, \ldots, 1)$ the dimension vectors for $\widetilde{R}_\ld$. They correspond to the partition $\lambda(\td) = \bar{1}$ (see \eqref{def:ldv}).
The advantage of manipulating over such modified quivers is that Proposition~\ref{prop:MS} applies, and hence it produces an isomorphism between the Nakajima quiver variety $M(\td, \tv)$ and the Slodowy variety $\tcS_{\bar{1}}$.

\rmk
Now we identify the linear maps $\tA_i, \tB_i, \tGa_1, \tDe_1$ as block matrices in light of \cite[(9)]{Maf05}, that is, for $1\leq i \leq n-2$ and suitable integers $a,b$, 
\eq\label{eq:TTSS}
\tGa_1
=
\begin{blockarray}{ *{8}{c} }
&  f_b & \dots & e_b \\
\begin{block}{ c @{\quad} ( @{\,} *{7}{c} @{\,} )}
V_1& \TT_{0,V}^{f,b}&\dots & \TT_{0,V}^{e,b}\\
f_a & \TT_{0,f,a}^{f,b}&\dots & \TT_{0,f,a}^{e,b}\\
\vdots & \vdots & \ddots & \vdots \\
e_a  & \TT_{0,e,a}^{f,b}&\dots& \TT_{0,e,a}^{e,b}  \\
\end{block}
\end{blockarray}
\normalsize
~~~~~~, 
\quad 
\tDe_1=
\small 
\begin{blockarray}{ *{8}{c} }
& V_1 &   f_b& \dots &e_b\\
\begin{block}{ c @{\quad} ( @{\,} *{7}{c} @{\,} ) }
f_a & \SS^V_{0,f,a}& \SS_{0,f,a}^{f,b}&\dots & \SS_{0,f,a}^{e,b}\\
\vdots & \vdots & \vdots & \ddots & \vdots \\
 e_a & \SS^V_{0,e,a} & \SS_{0,e,a}^{k,b}&\dots& \SS_{0,e,a}^{e,b}  \\
\end{block}
\end{blockarray}
\normalsize 
~~~~~~,
\endeq
\eq
\tA_i
=
\small 
\begin{blockarray}{ *{8}{c} }
&V_i&  f_b &e_b\\
\begin{block}{ c @{\quad} ( @{\,} *{7}{c} @{\,} ) }
V_{i+1}& \mathbb{A}_i& \TT_{i,V}^{f,b} & \TT_{i,V}^{e,b}\\
f_a &\TT_{i,f,a}^{V}& \TT_{i,f,a}^{f,b} & \TT_{i,f,a}^{e,b}\\
e_a &\TT_{i,e,a}^{V} & \TT_{i,e,a}^{f,b}& \TT_{i,e,a}^{e,b}  \\
\end{block}
\end{blockarray}
\normalsize 
~~~~~~, 
\quad 
\tB_i=
\small 
\begin{blockarray}{ *{8}{c} }
& V_{i+1} &  f_b &e_b\\
\begin{block}{ c @{\quad} ( @{\,} *{7}{c} @{\,} ) }
V_i& \mathbb{B}_i & \SS_{i,V}^{f,b} & \SS_{i,V}^{e,b}\\
f_a & \SS^V_{i,f,a}& \SS_{i,f,a}^{f,b} & \SS_{i,f,a}^{e,b}\\
e_a & \SS^V_{i,e,a} & \SS_{i,e,a}^{f,b}& \SS_{i,e,a}^{e,b}  \\
\end{block}
\end{blockarray}~~~~~~,
\endeq
\normalsize
with respect to the basis vectors (or subspaces when it is more convenient) indicated above and to the left of each matrix. 
For example, we have $\TT_{0,V}^{f,1} = \pi_{V_1}( \tGa_1 |_{\<f_1\>})$ is a linear map $\<f_1\> \to V_1$, where 
$\pi_{V_1}: D'_0 \to V_1 \oplus D'_1$ is the projection. 
In other words, the variables $\mathbb{A}, \mathbb{B}, \SS, \TT$ are certain linear maps with domains and codomains specified as below, for $\phi, \psi \in \{e,f\}$:
\eq
\begin{split}
&\mathbb{A}_i: V_i \to V_{i+1},
\quad
\mathbb{B}_i: V_{i+1}\to V_{i},
\\
&
\SS_{i,\phi,a}^V: V_{i+1} \to \<\phi_a\>,
\quad
\SS_{i,V}^{\phi,a}: \<\phi_a\> \to V_i,
\quad
\SS_{i,\phi,a}^{\psi, b}: \<\psi_b\> \to \<\phi_a\>,
\\
&
\TT_{i,\phi,a}^V: V_{i} \to \<\phi_a\>,
\quad
\TT_{i,V}^{\phi,a}: \<\phi_a\> \to V_{i+1},
\quad
\TT_{i,\phi,a}^{\psi, b}: \<\psi_b\> \to \<\phi_a\>.
\end{split}
\endeq
\endrmk
%=============================================
\subsection{Maffei's system of equations} 
%=============================================
%In our context it is convenient to rewrite the definition as below: 
Let $\tA_0 = \tGa_1, \tB_0 = \tDe_1$, 
and let $(x_i, y_i, [x_i, y_i])$ be the fixed $\fsl_2$-triple in $\fsl(D'_i)$ uniquely determined by
\eq\label{def:sl2}
\begin{aligned}
x_i(e_h) &= 
	\begin{cases}
	e_{h-1} &\tif 1< h \leq n-k-i, 
	\\
	\:\:\: 0 &\otw,
	\end{cases}
&&
y_i(e_h) = 
	\begin{cases}
	h(n-k-i-h)e_{h+1} &\tif 1\leq h < n-k-i, 
	\\
	\qquad \quad 0 &\otw,
	\end{cases}
\\
x_i(f_h) &= 
	\begin{cases}
	f_{h-1} &\tif 1< h \leq k-i, 
	\\
	\:\:\: 0 &\otw,
	\end{cases}
&&
y_i(f_h) = 
	\begin{cases}
	h(k-i-h) f_{h+1} &\tif 1\leq h < k-i, 
	\\
	\qquad \quad 0 &\otw.
	\end{cases}    
\end{aligned}
\endeq
\defn[cf. {\cite[Defn.~16]{Maf05}}] 
Let $\ld = (n-k,k)$.
An admissible quadruple $(\tA, \tB, \tGa, \tDe)$ in $\widetilde{R}_{\ld}$  is called {\em transversal} if the following system of equations holds ($0\leq i \leq n-2$):
\eq\label{eq:M1a}
\left[
\left(
\pi_{D'_i} \tB_i \tA_i|_{D'_i} -x_i \right), y_i
\right] = 0,
\endeq
\eq\label{eq:M1b}
\begin{aligned}
\TT^{f,b}_{i,f,a} = \TT^{e,b}_{i,e,a}  &= 0,        
&\tif b > a+1;
&&\SS^{f,b}_{i,f,a} = \SS^{e,b}_{i,e,a}  &= 0,        
&\tif b > a;  
	\\
\TT^{f,b}_{i,f,a} = \TT^{e,b}_{i,e,a}  &= \textup{id},          
&\tif b = a+1;
&&\SS^{f,b}_{i,f,a} = \SS^{e,b}_{i,e,a}  &= \textup{id},          
&\tif b = a;  
	\\
\TT^{e,b}_{i,f,a}  &= 0,         
&\tif b \geq a+1; 
&&\TT^{f,b}_{i,e,a}  &= 0,         
&\tif b \geq a+1+2k-n; 
	\\
\SS^{e,b}_{i,f,a}  &= 0,         
&\tif b \geq a; 
&&\SS^{f,b}_{i,e,a}  &= 0,         
&\tif b \geq a+2k-n;
	\\
\TT^{V}_{i,j,a}        &= 0;         
&&&   
\SS^{j,b}_{i,V}        &= 0;
	\\
\TT^{j,b}_{i,V}    &= 0,         
&\tif b \neq 1;
&&\SS^{V}_{i,j,a} &=0
&\tif a \neq j-i.
\end{aligned}
\endeq
\enddefn
Denote the subspace in $\widetilde{R}_{\ld}$ consisting of transversal (hence admissible) and stable elements by 
\eq\label{def:cT+}
\fT^+_{\ld} = \{(\tA, \tB, \tGa, \tDe)\in \widetilde{R}_{\ld}\mid
\textup{conditions }
\eqref{eq:M1a}, \eqref{eq:M1b} \eqref{eq:M2}, \eqref{eq:M3}
\textup{ hold} \},
\endeq
where the relations other than the transversal ones are
\begin{align}
\textup{(admissibility)}~&\tB_i\tA_i = \tA_{i-1} \tB_{i-1} + \tGa_i \tDe_i
\quad \mbox{for } 1\leq i\leq n-1, 
\label{eq:M2}
\\
\textup{(stability)}~&\Im \: \tA_{i-1} + \sum_{j\geq i} \Im \: \tGa_{j\to i} = \tV_i
\quad \mbox{for } 1\leq i\leq n-1,
\label{eq:M3}
\end{align}
where $\tGa_{j\to i}$ is defined similarly as \eqref{def:Gaij}.

\rmk
The system of equations \eqref{eq:M1b} is not the easiest to work with. 
For example, it implies that the map $\tGa_1$ must be of the following form:
\eq\label{eq:TTSS2}
\begin{tikzpicture}[baseline={(0,0)}, scale = 0.8]
\draw (1.2, 1.8) node {\scalebox{1.5}{$0$}};
\draw (-4.5, -.5) node {\scalebox{1.5}{$\TT_{0,e,\bullet}^{e,\bullet}$}};
\draw (7.2, 1.8) node {\scalebox{1.5}{$0$}};
\draw (-4.7, -2.7) node {\scalebox{1.5}{$\TT_{0,f,\bullet}^{e,\bullet}$}};
\draw (0, -2) node {\scalebox{1.5}{$0$}};
\draw (4.5, -.5) node {\scalebox{1.5}{$\TT_{0,e,\bullet}^{f,\bullet}$}};
\draw (8.2, -1.7) node {\scalebox{1.5}{$0$}};
\draw (4.5, -2.7) node {\scalebox{1.5}{$\TT_{0,f,\bullet}^{f,\bullet}$}};
  \draw (0,0) node {$
  \tGa_1
=
  \small 
\begin{blockarray}{ *{12}{c} }
&  e_1 & e_2 & \dots &\dots &\dots & e_{n-k} & f_1 & \dots & \dots & f_k 
\\
\begin{block}{ c @{\quad} ( @{\,} cccccc|ccccc @{\,} )}
V_1& \TT_{0,V}^{e,1}& 0 & \dots & \dots  & \dots& 0 & \TT_{0,V}^{f,1} & \dots & \dots & 0
\\
\cline{1-11}
e_1 & \TT_{0,e,1}^{e,1}& 1 &  & & &  &  &  &   & 
\\
\vdots & & \ddots & \ddots &  &  & &0 &  &   & 
\\
\vdots & &  & \ddots & \ddots &  & & & \ddots &   & 
\\
\vdots & &  &  & \ddots & \ddots &  & & & \ddots  & 
\\
e_{n-k-1} &  & &&& \TT_{0,e,n-k-1}^{e,n-k-1}&1 &  & && 0
\\
\cline{1-11}
f_1 & \TT_{0,f,1}^{e,1} &0 &&&&&\TT_{0,f,1}^{f,1} &1&&
\\
\vdots &  & \ddots &\ddots &&&& &\ddots&\ddots
\\
f_{k-1}
&&&\TT_{0,f,k-1}^{e,k-1}&0&&&&&\TT_{0,f,k-1}^{f,k-1}&1
\\
\end{block}
\end{blockarray}
\normalsize
~~~~.
$};
\end{tikzpicture}
\endeq
One can then solve for all the unknown variables $\TT_{0,\bullet,\bullet}^{\bullet,\bullet}$ using stability and admissibility conditions, together with the first transversality condition \eqref{eq:M1a}.
In general the solutions are very involved (see \cite[Lemma~18]{Maf05}.
We will use the proposition below to show that the solutions are actually very simple in our setup.
\endrmk
%--------------------------------------------------------------
\begin{prop}\label{prop:Phi}
Let $\ld = (n-k,k)$.
Let $(A, B, \Gamma, \Delta) \in \Lambda_{\ld}$. 
\begin{enumerate}[(a)]
\item There is a unique element $(\tA, \tB, \tGa, \tDe) \in \fT_{\ld}$ such that
\eq\label{eq:uniq}
\mathbb{A}_i = A_i,
\quad
\mathbb{B}_i = B_i,
\quad
\Gamma_k = \TT^{f,1}_{k-1,V},
~
\Gamma_{n-k} = \TT^{e,1}_{n-k-1,V},
\quad
\Delta_k = \SS^{V}_{k-1,f,1},
~
\Delta_{n-k} = \SS^{V}_{n-k-1,e,1}.
\endeq
\item The assignment in part (a) restricts to a $\GL(V)$-equivariant isomorphism $\Phi : \Lambda^+_{\ld} \overset{\sim}{\to} \fT^+_{\ld}$.
\item The map $\Phi$ induces an isomorphism $\Phi_M:M_{\ld} \overset{\sim}{\to} p_{\ld}(\fT^+_{\ld})$.
\end{enumerate}
\end{prop}
\proof
This is a special case of \cite[Lemmas~18, 19]{Maf05}.
\endproof
%--------------------------------------------------------------
%Background on GIT and stability: 
%\cite[Thm 1.1]{MFK94} and \cite[Thm 1.10]{MFK94},  
%\cite[Section 3.i]{Nak98}, 
%--------------------------------------------------------------
%--------------------------------------------------------------
Now we are in a position to define Maffei's isomorphism $\tphi_\ld: M_{\ld} \overset{\sim}{\to} \tcS_{\ld}$.
Recall $\Lambda^+_{\bullet}$ from \eqref{def:L+}, $p_{\bullet}$ from \eqref{def:p}, $M_{\bullet}$ from \eqref{def:M}, $ \tphi(\td,\tv)$ from Proposition~\ref{prop:MS} with codomain $\tcS_{\bar{1}}$ for the partition $\ld(\td) = \bar{1}$, $\fT^+_{\ld}$ from \eqref{def:cT+}.
Finally, $\tphi_{\ld}$ is defined such that the lower right corner of the diagram below commute:
\eq\label{def:tphi}
\begin{tikzcd}
\Lambda^+_{\bar{1}}
\ar[r, "{p_{\bar{1}}}", twoheadrightarrow] 
& 
M_{\bar{1}}
\ar[r, "{\tphi(\td,\tv)}", "\sim"'] 
&
\tcS_{\bar{1}}
\\
\fT^+_{\ld}
\ar[r, twoheadrightarrow] \ar[u, hookrightarrow]
& 
p_{\bar{1}}(\fT^+_{\ld})
\ar[r, "\sim"'] \ar[u, hookrightarrow]
&
\tphi(\td,\tv)\circ p_{\bar{1}}(\fT^+_{\ld}).
\ar[u, hookrightarrow]
\\
\Lambda^+_{\ld}
\ar[r, "p_{\ld}", twoheadrightarrow]
\ar[u, "\Phi", "\sim"']
&
M_{\ld}
\ar[r, "\tphi_{\ld}", dashrightarrow, "\sim"']
\ar[u, "\Phi_M", "\sim"']
&
\tcS_{\ld}
\ar[u, equal]
\end{tikzcd}
\endeq

%of the above quiver representations of type $(d,v)$ of the quiver of type $A_{n-1}$ is the vector space 
%\[ 
%R(d,v) = \bigoplus_{i=1}^{n-2} \Hom(V_i,V_{i+1}) 
%\oplus \bigoplus_{i=1}^{n-2} \Hom(V_{i+1},V_{i}) 
%\oplus \bigoplus_{i=1}^{n-1} \Hom(D_{i},V_{i}) 
%\oplus \bigoplus_{i=1}^{n-1} \Hom(V_{i},D_{i}).  
%\] 
%{\color{red}page $3$}. 

%=============================================

%=============================================

%=============================================
 
%=============================================

%Define the $n$-tuple $r=r(d,v)=(r_1,\dots,r_n)$ via 
%\begin{align*}
%	r_1 &= d_1+\dots+d_{n-1}-v_1,  \qquad \:\:r_{n}=v_{n-1}, \\
% 	r_i= d_i &+ \ldots+d_{n-1}-v_i+v_{i-1} \qquad \text{ for } i=2,\dots,n-1.
%\end{align*} 
%We may deduce that 
%$\sum_{i=1}^{n} r_i = N = \sum_{i=1}^{n-1} id_i$, 
%and,  
%for a fixed $d$, 
%the map $r$ gives a bijection between 
%$(n-1)$-tuples of integers $v$ and $n$-tuples of integers $r$ such that 
%$\sum_i r_i =N$.  
%Thus, 
%\[ 
%v_{n-1} = r_n, \qquad 
%v_i = r_n + \ldots + r_{i+1} - d_{i+1} - 2 d_{i+2}- \ldots - (n-i-1) d_{n-1}
%\] 
%for $i=1,\dots, n-2$.
%Recall that 
%$M(d,v)=M^1(d,v)=\varnothing$ if $v_i < 0$ 
%for some $i$ and 
%$\widetilde{\mathcal{S}}_{r,x} = \mathcal{S}_{r,x}=\varnothing$ if $r_i <0$ 
%for some $i$.
%The following was conjectured by \cite[Conj. 8.6]{Nak94} 
%--------------------------------------------------------------
\begin{prop} \label{prop:MS2}
The map $\tphi_{\ld}: M_{\ld} \to \tcS_{\ld}$ defined above is an isomorphism of algebraic varieties.
\end{prop}
 \proof
 This is a special case of a stronger result \cite[Thm. 8]{Maf05} for an arbitrary partition $\ld$. We are done by setting $\ld = (n-k,k)$, $N = n$, $r = (1, \ldots, 1)$, and $x\in\cN$ of Jordan type $\ld = (n-k,k)$.
 \endproof
%--------------------------------------------------------------

%=============================================

%=============================================
\section{Explicit descriptions of Maffei-Nakajima isomorphism} \label{sec:exp}
%=============================================

%=============================================
\subsection{Flag realization of quiver variety $M_{(n-k,k)}$} \label{sec:hell}
%=============================================
For now, let $\ld = (n-k,k)$. 
In the following we define auxiliary subspaces of $D'_0 = \<e_1, \ldots, e_{n-k}, f_1, \ldots, f_k\>$ using Young diagram combinatorics.
We first assign basis elements to boxes in the two-row Young diagram $\lambda$ as the following: 
\eq\label{eq:YT}
\ytableausetup{mathmode, boxsize=2.5em}
\begin{ytableau}
e_1 & e_2 &  \dots &  e_k &   \dots &\scriptstyle e_{n-k} \\
f_1 & f_2 &  \dots &  f_k 
\end{ytableau}
\endeq
For any subdiagram $\mu$ of $\lambda$, let $D'_\mu$ be the subspace of $D'_0$ spanned by elements in the corresponding boxes in $\mu$.
We denote by $\ld[i]$ the subdiagram of $\ld$ obtained by deleting the rightmost $i$ boxes for each row, i.e.,
\eq
\ld[i] = \begin{cases}
(n-k-i,k-i) 
&\tif i<k;
\\
(n-k-i) 
&\tif i\geq k.
\end{cases}
\endeq
Then the subspaces $D'_i$ (see \eqref{def:D'}) is equal to $D'_{\ld[i]}$ as in the following:
\eq
\ytableausetup{mathmode, boxsize=2.5em}
\begin{array}{ccc}
\begin{ytableau}
e_1 & e_2 &  \dots &  e_{k-i} &   \dots &\scriptstyle e_{n-k-i} \\
f_1 & f_2 &  \dots &  f_{k-i} 
\end{ytableau}
&
\begin{ytableau}
e_1 & e_2 &    \dots & \scriptstyle e_{n-k-i} \\
\end{ytableau}
\\~\\
D'_i = D'_{(n-k-i, k-i)} \quad (i < k)
&  D'_i = D'_{(n-k-i)} \quad (i \geq k) 
\end{array}
\endeq
For a subdiagram $\mu = (\mu_1, \mu_2)$ of $\lambda$ that does not contain the rightmost box in any row of $\lambda$, we denoted by $\mu^+$ the skew diagram obtained by shifting $\mu$ one row to the right, i.e.,
\eq\label{def:+}
\mu^+ = (\mu_1 +1, \mu_2 +1)/(1,1).
\endeq
In order to describe the explicit solutions, we will need two Young subdiagrams $\ld[i]/ \ld[i+1]$ and $\ld[i]/ \ld[i+1]^+$ of $\lambda$. The corresponding subspaces are
\eq
D'_{\ld[i] / \ld[i+1]} =\begin{cases}
\<e_{n-k-i}, f_{k-i}\> &\tif i < k;
\\
\<e_{n-k-i}\> &\tif i \geq k,
\end{cases} 
\quad
D'_{\ld[i]/ \ld[i+1]^+} =\begin{cases}
\<e_1,f_1\> &\tif i < k;
\\
\<e_1\> &\tif i \geq k.
\end{cases}
\endeq
%For $1\leq i \leq n-1$, we define an isomorphic copy of $D'_i$ (see \eqref{def:D'}) with a shift of index by $t$
%as
%\eq
%D'_i[t] = \<f_{j+t}\mid f_j \in D'_i\>  \oplus \< e_{j+t} \mid e_j \in D'_i\>.
%\endeq
By a slight abuse of notation, we denote by $\Gamma_{\to i}$ the linear assignment (without specifying its domain) given by
\eq
e_\bullet \mapsto \Gamma_{n-k\to i}(e) \in V_i,
\quad
f_\bullet \mapsto \Gamma_{k\to i}(f) \in V_i.
\endeq
For example, both linear maps below are denoted by $\Gamma_{\to 1}$:
\eq
\begin{split}
&
D'_{\ld/\ld[1]^+} = \<e_1, f_1\> \to V_1, \quad \mu e_1 + \nu f_1 \mapsto 
\mu\Gamma_{n-k\to1}(e)
+
\nu  \Gamma_{k \to 1}(f),
\\
&
D'_{\ld[k]/\ld[k+1]^+} = \<e_{n-2k}\> \to V_1, \quad \mu e_{n-2k} \mapsto \mu\Gamma_{n-k\to1}(e).
\end{split}
\endeq
We define the obvious composition by
\eq\label{def:Deij}
\Delta_{j\to i} = 
\begin{cases} 
\Delta_i B_i  \cdots B_{j-1}  &\tif j \geq i, 
\\
\Delta_i A_{i-1} \cdots A_j  &\tif j \leq i,
\end{cases}
\endeq
and then define $\Delta_{i\to }$ similarly.
%--------------------------------------------------------------
\begin{lemma}[Explicit solutions to Maffei's system]\label{lem:hell}
Let $(\tA, \tB, \tGa, \tDe) \in \fT^+_{(n-k,k)}$ and 
$\Phi\inv(\tA, \tB, \tGa, \tDe) = (A,B,\Gamma, \Delta) \in \Lambda^+_{(n-k,k)}$.
Moreover, assume that the following equation hold: 
\eq\label{eq:T}
\Delta_{i \to j} \Gamma_{j\to i} = 0
\quad
(2\leq i \leq j \leq n-1).
\endeq
Then $(\tA, \tB, \tGa, \tDe)$ is determined by, for $1\leq i \leq n-2$: 
\eq
\begin{split}
\tGa_1:& 
f_{k} \mapsto f_{k-1} \mapsto \ldots \mapsto f_2 \mapsto f_1 \mapsto \Gamma_{k \to 1}(f),
\quad
e_{n-k} \mapsto  \ldots \mapsto e_1 \mapsto \Gamma_{n-k \to 1}(e),
\\
\tDe_1:
&v \mapsto \nu e_{n-k} + \mu f_k ~\tif~ \Delta_{1\to n-k}(v) = \nu e, \Delta_{1\to k}(v) = \mu f,
\quad
\textup{for}
\quad 
v\in V_1,
\\
& e_j \mapsto e_j, 
\quad 
f_j \mapsto f_j
\quad
\textup{for all}
\quad
j,
\\
\tA_i: &v \mapsto A_i(v) \quad \textup{for} \quad v \in V_i,
\\
&f_{k-i} \mapsto \ldots \mapsto f_1 \mapsto \Gamma_{k \to i+1}(f),
\quad
e_{n-k-i} \mapsto  \ldots \mapsto e_1 \mapsto \Gamma_{n-k \to i+1}(e),
\\
\tB_i: &v \mapsto B_i(v) + \nu e_{n-k-i} + \mu f_{k-i} ~\tif~ \Delta_{i+1\to n-k}(v) = \nu e, \Delta_{i+1\to k}(v) = \mu f,
\quad
\textup{for}
\quad 
v\in V_{i+1},
\\
& e_j \mapsto e_j, 
\quad 
f_j \mapsto f_j
\quad
\textup{for all}
\quad
j.
\end{split}
\endeq

It is convenient to visualize the linear maps as block matrices as below:
\eq
\tGa_1
=
\begin{blockarray}{ *{8}{c} }
&  D'_{\ld[0] / \ld[1]^+} & D'_{\ld[1]^+} \\
\begin{block}{ c @{\quad} ( @{\,} *{7}{c} @{\,} )}
V_1& \Gamma_{\to1}&0\\
D'_1 & 0& I & \\
\end{block}
\end{blockarray}
\normalsize
~~~~~, 
\quad 
\tDe_1=
\small 
\begin{blockarray}{ *{8}{c} }
& V_1 &  D'_1\\
\begin{block}{ c @{\quad} ( @{\,} *{7}{c} @{\,} ) }
D'_1 & 0 & I \\
D'_{\ld[0]/\ld[1]}& \Delta_{1\to} &   0\\
\end{block}
\end{blockarray}
\normalsize 
~~~~~,
\endeq
\eq
\tA_i
=
\begin{blockarray}{ *{8}{c} }
& V_i & D'_{\ld[i]/\ld_{i+1}^+} & D'_{\ld[i+1]^+} \\
\begin{block}{ c @{\quad} ( @{\,} *{7}{c} @{\,} )}
V_{i+1}& {A}_i&\Gamma_{\to i+1} & 0\\
D'_{i+1} & 0 &0 &I   \\
\end{block}
\end{blockarray}
\normalsize
~~~~~,
\quad
\tB_i=
\small 
\begin{blockarray}{ *{8}{c} }
& V_{i+1} &   D'_{i+1}\\
\begin{block}{ c @{\quad} ( @{\,} *{7}{c} @{\,} ) }
V_i & {B}_i & 0 \\
D'_{i+1}& 0 &   I   \\
D'_{\ld[i]/\ld[i+1]} & \Delta_{i+1 \to} &0\\
\end{block}
\end{blockarray}
\normalsize 
~~~~~.
\endeq
Here $I$ represents the linear map sending $e$'s to $e$'s and $f$'s to $f$'s with suitable subscripts.
\end{lemma}
%--------------------------------------------------------------
\proof
By Proposition~\ref{prop:Phi} we know that $(\tA,\tB,\tGa,\tDe)$ must be the unique element in $\fT^+_{\ld}$ that satisfies \eqref{eq:uniq}, which indeed are satisfied from the construction.
What remains to be shown is that the formulas above do define an element in $\fT^+_{\ld}$.

By construction, we have
\eq\label{eq:tDG}
\tGa_1 \tDe_1=
\small 
\begin{blockarray}{ *{8}{c} }
& V_{1} & D'_{\ld[1]/\ld[2]^+} & D'_{\ld[2]^+}\\
\begin{block}{ c @{\quad} ( @{\,} *{7}{c} @{\,} ) }
V_1 & 0 & \Gamma_{\to1} & 0 
\\
D'_2& 0 & 0&  I   \\
D'_{\ld[1]/\ld[2]} &\Delta_{1\to} &0&0\\
\end{block}
\end{blockarray}
\normalsize 
~~~~,
\quad
\tDe_1 \tGa_1=
\small 
\begin{blockarray}{ *{8}{c} }
&  D'_{\ld /\ld[1]^+}& D'_{\ld[1]^+}\\
\begin{block}{ c @{\quad} ( @{\,} *{7}{c} @{\,} ) }
D'_1& 0&  I   \\
D'_{\ld/\ld[1]}& \Delta_{1\to}\Gamma_{\to 1} &0\\
\end{block}
\end{blockarray}
\normalsize 
~~~~~,
\endeq
\eq\label{eq:tAtB}
\tA_i \tB_i=
\small 
\begin{blockarray}{ *{8}{c} }
& V_{i+1} & D'_{\ld[i+1]/\ld[i+2]^+}& D'_{\ld[i+2]^+}\\
\begin{block}{ c @{\quad} ( @{\,} *{7}{c} @{\,} ) }
V_{i+1} & A_iB_i & \Gamma_{\to i+1}  & 0 
\\
D'_{i+2}& 0& 0&  I   \\
D'_{\ld[i+1]/\ld[i+2]} & \Delta_{i+1\to} &0&0\\
\end{block}
\end{blockarray}
\normalsize 
~~~~~~,
\endeq
\eq\label{eq:tBtA}
\tB_i \tA_i=
\small 
\begin{blockarray}{ *{8}{c} }
& V_i & D'_{\ld[i]/\ld[i+1]^+}& D'_{\ld[i+1]^+}\\
\begin{block}{ c @{\quad} ( @{\,} *{7}{c} @{\,} ) }
V_i& B_i A_i & \Gamma_{\to i}  & 0\\
D'_{i+1}& 0&0&  I   \\
D'_{\ld[i] /\ld[i+1]} &\Delta_{i\to} &\Delta_{i+1\to}\Gamma_{\to i+1}&0\\
\end{block}
\end{blockarray}
\normalsize 
~~~~~~.
\endeq
The admissibility conditions \eqref{eq:M2} follow from comparing the entries in \eqref{eq:tAtB} -- \eqref{eq:tBtA} with the original admissibility condition \eqref{eq:L2} and \eqref{eq:T}. 

By the original stability condition \eqref{eq:L1}, the blocks in the first row for $\tA_i$ have full rank, and hence \eqref{eq:M3} follows.

From \eqref{eq:tBtA} we see that
\eq
\pi_{D'_i} \tB_i \tA_i|_{D'_i} -x_i
=
\small 
\begin{blockarray}{ *{8}{c} }
&  D'_{\ld[i]/\ld[i+1]^+} & D'_{\ld[i+1]^+}\\
\begin{block}{ c @{\quad} ( @{\,} *{7}{c} @{\,} ) }
D'_{i+1}&0&  0   \\
D'_{\ld[i]/\ld[i+1]} &\Delta_{i+1\to}\Gamma_{\to i+1}&0\\
\end{block}
\end{blockarray}
\normalsize 
~~~~~~, 
\endeq
and thus \eqref{eq:M1a} holds.
Finally, a straightforward verification similar to \eqref{eq:TTSS2} shows that \eqref{eq:M1b} is satisfied.
\endproof
%--------------------------------------------------------------
We can now combine Lemma~\ref{lem:hell} with the Nakajima-Maffei isomorphism to describe the complete flag assigned to each  quiver representation. 
%-----------------------------------------------------------------
\begin{thm}[Explicit flag realization of $M_{(n-k,k)}$]\label{thm:Fi}
If $[A,B,\Gamma,\Delta] = \tphi_{(n-k,k)}\inv(x, F_\bullet)$ is a geometric point in $M_{(n-k,k)}$ satisfying \eqref{eq:T},
then we have an explicit description for the flag $F_\bullet$ under $\tphi_{(n-k,k)}$.
Precisely, each $F_i$ is the kernel of the map $D''_0 \oplus \dots \oplus D''_{i-1} \to V_i$ whose blocks are described as below:
\[
\begin{blockarray}{ *{8}{c} }
 & D''_1 & \dots & D''_{t} & \dots & D''_{i} \\
\begin{block}{ c @{\quad} ( @{\,} *{7}{c} @{\,} )}
V_{i}& 
A_{1\to i}\Gamma_{\to 1}
&
\dots
&
A_{t\to i}\Gamma_{\to t}
&
\dots
&\Gamma_{\to i} \\
\end{block}
\end{blockarray}
\normalsize
~~~~~,
\]
where $D''_{t}$, for $1\leq t \leq i$, is the space spanned by  elements in the $t$th column in \eqref{eq:YT} as below:
\eq
D''_{t} %= D'_{t-1}[t-1] - D'_{t}[t]
=
\begin{cases}
\<e_{t}, f_{t}\> &\tif t \leq k, 
\\
\hspace{3mm} \<e_{t}\> &\tif k+1 \leq t \leq n-k.
\end{cases}
\endeq
\end{thm}
\proof
By Proposition~\ref{prop:MS} we know that the spaces $F_i$ are determined by the kernels of the maps $\tGa_{1\to i}$.
The assertion follows from a direct computation of $\tGa_{1\to i}$ using Lemma~\ref{lem:hell}, which is,
\eq\label{eq:FiKer}
\tGa_{1\to i}
=
\begin{blockarray}{ *{8}{c} }
 & D''_1 & \dots & D''_{t} & \dots & D''_{i} & D'_{(\ld_1, \ld_2)/(i,i)} \\
\begin{block}{ c @{\quad} ( @{\,} *{7}{c} @{\,} )}
V_{i}& 
A_{1\to i}\Gamma_{\to 1}
&
\dots
&
A_{t\to i}\Gamma_{\to t}
&
\dots
&\Gamma_{\to i} & 0\\
D'_{i}  &0 &\dots & 0 &
\dots & 0 &I   \\
\end{block}
\end{blockarray}
\normalsize
~~~~~.
\endeq
\endproof

It is known (see Proposition~\ref{prop:Spr}) that in any quiver representation corresponding to a geometric point in a Springer fiber, all the $\Delta_i$'s are zero maps , and hence \eqref{eq:T} is automatically satisfied. 

%=============================================
\subsection{Examples}
%=============================================

In this section we provide examples demonstrating how Theorem~\ref{thm:Fi} provides an efficient and explicit way to realize the entire quiver variety $M_{(n-k,k)}$ using the corresponding complete flags in the Slodowy variety $\tcS_{(n-k,k)}$, while generally it is very implicit to apply the Maffei--Nakajima isomorphism.

It can also be seen in Example~\ref{ex:S31} that $\ker \tGa_{1\to i}$ is not  {necessarily} a direct sum of the kernels of the block matrices.

\ex\label{ex:Fi}
The quiver representations in $R_{(2,2)}$ are described as below:
\begin{equation}\label{eq:22-quiver} 
\xymatrix@C=18pt@R=9pt{
&
\ar@/_/[dd]_<(0.2){\Gamma_2}
D_2 =\<e,f\>
&   
\\   
& &    
\\ 
V_1 =\CC
	\ar@/^/[r]^{A_1}
	& 
	\ar@/^/[l]^{B_1} 
V_2  =\CC^2
	\ar@/_/[uu]_>(0.8){\Delta_2} 
	\ar@/^/[r]^{A_{2}}
	& 
	\ar@/^/[l]^{B_2}
V_3 = \CC. 
	}
\end{equation} 
Let $[A,B,\Gamma,\Delta] = \tphi_{(2,2)}\inv(x, F_\bullet)  \in M_{(2,2)}$.
By Theorem~\ref{thm:Fi}, the  flag $F_\bullet$ is described by,
\[
\begin{split}
F_1 &=  
 \ker \Big(\begin{blockarray}{ *{8}{c} }
 & \<f_1, e_1\>  \\
\begin{block}{ c @{\quad} ( @{\,} *{7}{c} @{\,} )}
V_{1}& 
B_1\Gamma_2
\\
\end{block}
\end{blockarray}
\Big)
\normalsize
~~~~~,
\\
F_2 &= 
 \ker \Big(\begin{blockarray}{ *{8}{c} }
 & \<f_1,e_1\> &\<f_2,e_2\>  \\
\begin{block}{ c @{\quad} ( @{\,} *{7}{c} @{\,} )}
V_{2}& 
A_1B_1\Gamma_2
&
\Gamma_2
\\
\end{block}
\end{blockarray}
\Big)
\normalsize
~~~~~,
\\
F_3 &= 
 \ker \Big(\begin{blockarray}{ *{8}{c} }
 & \<f_1,e_1\> &\<f_2,e_2\>  \\
\begin{block}{ c @{\quad} ( @{\,} *{7}{c} @{\,} )}
V_{3}& 
A_2A_1B_1\Gamma_2
&
A_2\Gamma_2
\\
\end{block}
\end{blockarray}
\Big)
\normalsize
~~~~~. 
\end{split}
\]
In particular, if 
$
A_1 = \begin{pmatrix}1 \\ 0\end{pmatrix} = B_2,
%\quad
A_2 = \begin{pmatrix} 0 & 1 \end{pmatrix} = B_1,
%\quad
\Gamma_2 = \begin{pmatrix}1&0\\0&1\end{pmatrix},
%\quad
\Delta = 0,
$
then
\[
F_1 = \ker \begin{pmatrix} 0 & 1 \end{pmatrix} = \<f_1\>,
\quad
F_2 = \ker \begin{pmatrix} 0 & 1 &1& 0 \\ 0&0&0&1 \end{pmatrix} = \<f_1, e_1 - f_2\>,
\quad
F_3 = \ker \begin{pmatrix} 0 & 0 &0& 1  \end{pmatrix} = \<f_1, e_1, f_2\>.
\]
\endex
\ex\label{ex:S31}
The quiver representations in $R_{(3,1)}$ are described as below:
\[
 \xymatrix@-1pc{
D_1=\<f\> \ar@/^/[dd]^{\Gamma_1} &&& &   D_3=\<e\>  \ar@/^/[dd]^{\Gamma_3.}  \\ 
 & & \\ 
V_1=\CC\ar@/^/[uu]^{\Delta_1} \ar@/^/[rr]^{A_1} 
& & \ar@/^/[ll]^{B_1}  V_2=\CC  \ar@/^/[rr]^{A_2}   
& & \ar@/^/[ll]^{B_2}  V_3=\CC  \ar@/^/[uu]^{\Delta_3}\\ 
 }
\] 
If $[A,B,\Gamma,\Delta] = \tphi_{(3,1)}\inv(x, F_\bullet) \in M_{(3,1)}$, then the flag $F_\bullet$ can be described by
\[
\begin{split}
F_1 &= 
 \ker \Big(\begin{blockarray}{ *{8}{c} }
 & \<f_1\> &\<e_1\>  \\
\begin{block}{ c @{\quad} ( @{\,} *{7}{c} @{\,} )}
V_{1}& 
\Gamma_1
&
B_1B_2\Gamma_3
\\
\end{block}
\end{blockarray}
\Big)
\normalsize
~~~~~,
\\
F_2 &
= \ker \Big(\begin{blockarray}{ *{8}{c} }
 & \<f_1\> & \<e_1\>  & \<e_2\> \\
\begin{block}{ c @{\quad} ( @{\,} *{7}{c} @{\,} )}
V_{2}&
A_1\Gamma_1
& 
A_1B_1B_2\Gamma_3
&
B_2\Gamma_3
\\
\end{block}
\end{blockarray}
\Big)
\normalsize
~~~~~,
\\
F_3 & =  \ker\Big( \begin{blockarray}{ *{8}{c} }
 & \<f_1\>& \<e_1\>  & \<e_2\> & \<e_3\>\\
\begin{block}{ c @{\quad} ( @{\,} *{7}{c} @{\,} )}
V_{3}& 
A_2A_1\Gamma_1
&
A_2A_1B_1B_2\Gamma_3
&
A_2B_2\Gamma_3
&
\Gamma_3
\\
\end{block}
\end{blockarray}
\Big)
\normalsize
~~~~~.
\end{split}
\]
In particular, if 
$
A_1 = 1, A_2 =0,
B_1 = 0, B_2 =1,
\Gamma_1 = 1, \Gamma_3 = 1,
\Delta_1 = 0 = \Delta_3,
$
then
\[
F_1 = \ker \begin{pmatrix} 1 & 0 \end{pmatrix} = \<e_1\>,
\quad
F_2 = \ker \begin{pmatrix} 1 & 0 & 1  \end{pmatrix} = \<e_1, f_1 - e_2\>,
\quad
F_3 = \ker \begin{pmatrix} 0 & 0 &0& 1  \end{pmatrix} = \<e_1, f_1, e_2\>.
\]
\endex

%=============================================
\subsection{Explicit flag realization of quiver varieties beyond two rows} \label{sec:rrows}
%=============================================
From now on, we fix a partition $\ld = (\ld_1, \ld_2, \ldots, \ld_r)$ that can be more than two rows.
We will generalize Lemma~\ref{lem:hell} to this arbitrary $\ld$, and then obtain a flag realization of any geometric point in $M_\ld$.

Fix a basis $\{t_{i,j}~|~ 1\leq i \leq r, 1\leq j \leq \ld_i\}$ of $\CC^n$. Treat $t_{i,j}$ as zero if $(i,j)$ is not in the Young diagram $\ld$. 
Fix a nilpotent map $x$ with Jordan type $\ld$ that takes each $t_{i,j}$ to $t_{i,j-1}$. 

For $1\leq i \leq r$, denote the space $D_i$ for the shadow vertex by
\eq
D_i = \< t_j ~|~ \ld_j = i\>.
\endeq
Note that $t_i \in D_{\ld_i}$, and  $d_i = \dim D_i = \sum_{j=1}^r \delta_{i,\ld_j}$. 
Hence, $v_i$'s are uniquely determined by \eqref{eq:dv}. 
Let $V_i$ be a $v_i$-dimensional space which we don't need to specify a basis.

%Combinatorially, if we identify basis element $t_{i,j}$ with the $(i,j)$th box in the partition $\ld$ as the following tableau to the left, $D'_l$ is then the subspace spanned by elements corresponding to  the subdiagram to the right below:
We assign basis elements to boxes in the Young diagram $\lambda$ as the following: 
\eq\label{eq:YT2}
\ytableausetup{mathmode, boxsize=2.5em}
\begin{ytableau}
t_{1,1} & t_{1,2} &  \dots &  \dots &   \dots & t_{1,\ld_1} \\
t_{2,1} & t_{2,2} &  \dots &  t_{2,\ld_2}\\
\vdots & \vdots & \vdots \\
t_{r,1} & \dots & t_{r,\ld_r} 
\end{ytableau}
\endeq
Note that the definition of $D'_\mu$ and $\ld[i]$ in Section~\ref{sec:hell} naturally extends beyond two rows, and hence 
\eq
D'_l = D'_{\ld[l]}  = \< t_{i,j} ~|~1 \leq i \leq r, 1 \leq j \leq \ld_i - l\>.
\endeq

The modified quiver representation space now consists of quadruples $(\tA, \tB, \tGa, \tDe)$ of linear maps of the form
\eq
\tGa_1: D'_0 \to V_1\oplus D'_1, 
\quad
\tA_l: V_l \oplus D'_l \to V_{l+1} \oplus D'_{l+1},
\quad
\tDe_1:V_1\oplus D'_1 \to D'_0,
\quad
\tB_l: V_{l+1} \oplus D'_{l+1} \to V_l \oplus D'_l.
\endeq

For a fixed $v \in V_l$, note that for all $t$, $\Delta_{j \to t}(v) \in D_t$  can be expressed as a linear combination of $t_i$'s which do not belong to $D_u$ if $u \neq t$.
Therefore, we can define coefficients $\nu_{i,j} = \nu_{i,j}(v)$ such that 
\eq\label{eq:nuij}
\sum_{t} \Delta_{j\to t}(v) = \sum_i \nu_{ij} t_i.
\endeq
%--------------------------------------------------------------
\begin{lemma}[Explicit solutions to Maffei's system]\label{lem:hell2}
Let $\lambda$ be an arbitrary partition of $n$, and
let $(\tA, \tB, \tGa, \tDe) \in \fT^+_{\ld}$ and $\Phi\inv(\tA, \tB, \tGa, \tDe) = (A,B,\Gamma, \Delta) \in \Lambda^+_{\ld}$.
Moreover, assume that \eqref{eq:T} holds. 
Then $(\tA, \tB, \tGa, \tDe)$ is given by, for $1\leq l \leq n-2$: 
\eq
\begin{split}
\tGa_1:& 
t_{i,j} \mapsto \begin{cases}
\Gamma_{\ld_i \to 1}(t_i) &\tif j=1;
\\
t_{i,j-1} &\otw,
\end{cases}
\\
\tA_l: &v \mapsto A_l(v) \quad \textup{for} \quad v \in V_l,
\\
&t_{i,j} \mapsto \begin{cases}
\Gamma_{\ld_i \to l+1}(t_i) &\tif j=1;
\\
t_{i,j-1} &\otw,
\end{cases}
\\
\tDe_1:
&v \mapsto \sum_i \nu_{i,1} t_{i,\ld_i}  \quad t_{i,j} \mapsto t_{i,j},
\\
\tB_l: &v \mapsto B_l(v)+\sum_i \nu_{i,l+1} t_{i,\ld_i-l}, \quad t_{i,j} \mapsto t_{i,j}.
\end{split}
\endeq
\end{lemma}
%--------------------------------------------------------------
\proof
The proof is almost identical to the one for Lemma~\ref{lem:hell}. All we need are the following identities which follow from a lengthy but elementary computation:
\eq
\begin{split}
\tGa_1 \tDe_1:& v \mapsto \sum_{i: \ld_i > 1} \nu_i t_{i,\ld_i-1}
+\sum_{i: \ld_i = 1} \nu_{i,1} \Gamma_{\ld_i \to 1}(t_i), 
 \quad t_{i,j} \mapsto \begin{cases}
\Gamma_{\ld_i \to 1}(t_i) &\tif j=1;
\\
t_{i,j-1} &\otw,
\end{cases}
\\
\tDe_1\tGa_1 :& t_{i,j} \mapsto \begin{cases}
\Delta_{1\to}\Gamma_{ \to 1}(t_{i,1}) &\tif j=1;
\\
t_{i,j-1} &\otw,
\end{cases}
\\
\tA_l\tB_l: &v \mapsto A_l B_l (v) + \sum_{i: \ld_i > 1} \nu_{i,l+1} t_{i,\ld_i-1}
+\sum_{i: \ld_i = 1} \nu_{i,l+1} \Gamma_{\ld_i \to l+1}(t_i),
\\
&t_{i,j} \mapsto \begin{cases}
\Gamma_{\ld_i \to l+1}(t_i) &\tif j=1;
\\
t_{i,j-1} &\otw,
\end{cases}
\\
\tB_l \tA_l: &v \mapsto  B_l A_l(v) + \sum_{i: \ld_i > 1} \nu_{i,l} t_{i,\ld_i-1}
+\sum_{i: \ld_i = 1} \nu_{i,l} \Gamma_{\ld_i \to l}(t_i),
\\
&t_{i,j} \mapsto \begin{cases}
\Gamma_{\ld_i \to l}(t_i) &\tif j=1;
\\
t_{i,j-1} &\otw.
\end{cases}
\end{split}
\endeq
\endproof
%=============================================
%\subsection{Flag realization of quiver representations}
%=============================================
Now we can once again use Proposition~\ref{prop:MS} to obtain a flag realization of $M_\ld$ in light of Lemma~\ref{lem:hell2}.
%-----------------------------------------------------------------
\begin{thm}[Explicit flag realization of $M_\ld$]\label{thm:Fi2}
Let $\lambda$ be an arbitrary partition of $n$.
If $[A,B,\Gamma,\Delta] = \tphi_{\ld}\inv(x, F_\bullet)$ is a geometric point in $M_\ld$ satisfying \eqref{eq:T},
then we have an explicit description for the flag $F_\bullet$ under $\tphi$.
Precisely, each $F_l$ is the kernel of the map $D''_1 \oplus \dots \oplus D''_{l} \to V_l$ whose blocks are described as below:
\[
\begin{blockarray}{ *{8}{c} }
 & D''_1 & \dots & D''_{t} & \dots & D''_{l} \\
\begin{block}{ c @{\quad} ( @{\,} *{7}{c} @{\,} )}
V_{l}& 
A_{1\to l}\Gamma_{\to 1}
&
\dots
&
A_{t\to l}\Gamma_{\to t}
&
\dots
&\Gamma_{\to l} \\
\end{block}
\end{blockarray}
\normalsize
~~~~~,
\]
where $D''_{j} = \< t_{i,j} ~|~ 1\leq i \leq r\>$ is the space spanned by  elements in the $j$th column in \eqref{eq:YT2}.
\end{thm}
\proof
It suffices to compute the kernel of $\tGa_{1\to l}$.
By Lemma~\ref{lem:hell2} we have $\tGa_{1\to l}(t_{ij}) = A_{1\to l}\Gamma_{\ld_i \to 1} (t_i)$, and we are done.
\endproof

\ex\label{ex:Fi211}
Let $n=4$. Consider the partition $\ld = (2,1,1)$ with $d_1 = 2, d_2 =1$ and $d_i = 0$ for $i \geq 3$. By \eqref{eq:dv} we know that
\eq
v_1 = 3 - 2d_3 - d_2 = 2,
\quad
v_2 = 2 - d_3 = 2,
\quad
v_3 = 1.
\endeq
Therefore, any geometric point in $M_\ld$ is represented by a quiver representation  as below:
\begin{equation}\label{eq:211-quiver} 
\xymatrix@C=18pt@R=9pt{
\ar@/_/[dd]_<(0.2){\Gamma_1}
D_1 =\<t_2, t_3\>
&
\ar@/_/[dd]_<(0.2){\Gamma_2}
D_2 =\<t_1\>
&   
\\   
& &    
\\ 
V_1 =\CC^2
	\ar@/^/[r]^{A_1}
	\ar@/_/[uu]_>(0.8){\Delta_1} 
	& 
	\ar@/^/[l]^{B_1} 
V_2  =\CC^2
	\ar@/_/[uu]_>(0.8){\Delta_2} 
	\ar@/^/[r]^{A_{2}}
	& 
	\ar@/^/[l]^{B_2}
V_3 = \CC. 
	}
\end{equation} 
Let $[A,B,\Gamma,\Delta] = \tphi_{\ld}\inv(x, F_\bullet)  \in M_{\ld}$.
By Theorem~\ref{thm:Fi2}, the  flag $F_\bullet$ is described by,
\[
\begin{split}
F_1 &=  
 \ker \Big(\begin{blockarray}{ *{8}{c} }
 & \< t_{2,1}, t_{3,1}\>&\<t_{1,1}\>  \\
\begin{block}{ c @{\quad} ( @{\,} *{7}{c} @{\,} )}
V_{1}& 
\Gamma_1 &B_1\Gamma_2
\\
\end{block}
\end{blockarray}\Big)
\normalsize
~~~~~,
\\
F_2 &= 
 \ker\Big( \begin{blockarray}{ *{8}{c} }
 & \< t_{2,1}, t_{3,1}\>&\<t_{1,1}\> & \<t_{1,2}\> \\
\begin{block}{ c @{\quad} ( @{\,} *{7}{c} @{\,} )}
V_{2}& 
A_1\Gamma_1 &A_1B_1\Gamma_2& \Gamma_2
\\
\end{block}
\end{blockarray}
\Big)\normalsize
~~~~~,
\\
F_3 &= 
 \ker \Big(\begin{blockarray}{ *{8}{c} }
 & \< t_{2,1}, t_{3,1}\>&\<t_{1,1}\> & \<t_{1,2}\> \\
\begin{block}{ c @{\quad} ( @{\,} *{7}{c} @{\,} )}
V_{3}& 
A_2A_1\Gamma_1 &A_2A_1B_1\Gamma_2& A_2\Gamma_2
\\
\end{block}
\end{blockarray}
\Big)\normalsize
~~~~~. 
\end{split}
\]
In particular, if 
$
A_1 = \Gamma_1 = I_2,
\Delta_1 = B_1 = B_2 = \Gamma_2 = 0,
A_2 = \Delta_2 = (1 ~1),
$
then
\[
\begin{split}
&F_1 = \ker \begin{pmatrix} 1& 0 & 0 \\ 0&1&0 \end{pmatrix} = \<t_{1,1}\>,
\quad
F_2 = \ker \begin{pmatrix} 1 & 0 &0& 0 \\ 0&1&0&0 \end{pmatrix} = \<t_{1,1}, t_{1,2}\>,
\\
&
F_3 = \ker \begin{pmatrix} 1 & 1 &0&  0  \end{pmatrix} = \<t_{1,1}, t_{1,2}, t_{2,1}-t_{3,1}\>.
\end{split}
\]
Since $\Delta_2$ is nozero, we have computed the corresponding flag for a geometric point in $M_\ld$ that is not in the Lagrangian subvariety corresponding to the Springer fiber $\cB_\ld$. 
\endex
%=============================================
\section{Applications}\label{sec:appl} 
%=============================================
In this section we talk about a few applications of our main theorem (Theorem~\ref{thm:Fi}) in the two-row case. We describe the irreducible components of the Lagrangian subvarieties (of type A), which leads to a characterization of irreducible components of Springer fibers of classical type thanks to the theory of fixed-point subvarieties (or $\imath$quiver varieties).

%-------------------------------------------------------------

%=============================================
\subsection{Irreducible components of two-row Springer fibers}\label{sec:cupdiag}
%=============================================
We assume that $\ld = (n-k,k)$ for the rest of article.
A {\em cup diagram} is a non-intersecting arrangement of cups and rays below a horizontal axis connecting a subset of $n$ vertices on the axis, and we identify any cup diagram $a$ with the collection of ordered pairs of endpoints of cups as below:
\eq\label{eq:cup}
a \equiv \{ (i_t, j_t) ~|~ 1\leq t \leq k, 1\leq i_t < j_t \leq n \} 
\quad
\textup{for some}
\quad
k \leq \left\lfloor\frac{n}{2}\right\rfloor.
\endeq
By a cup we mean an ordered pair $(i,j)$ such that $i<j$.
We use a rectangular region that contains all the cups to represent the cup diagram. 
Note that a ray is now a through-strand in this presentation, but we still call it a ray.

The irreducible components of the Springer fiber $\mathcal B_\ld$  can be labeled by the set $I_{\ld}$ of all cup diagrams on $n$ vertices with $k$ cups and $n-2k$ rays.
For example, when $n=3$ and $k=1$ we have
\eq
%I_{(3,0)} = \left\{ \cupc \right\},
%\quad
I_{(2,1)} = \left\{  \cupa~,~ \cupb \right\}.
\endeq
It is known in \cite{Spa76} (cf. \cite{Var79}) that irreducible components of $\cB_{\ld}$ are indexed by $I_{\ld}$.
Denote by $K_a$ the irreducible component in $\mathcal B_{\ld}$ associated to the cup diagram $a\in I_{\ld}$ as in {\em loc. cit}.

%Now we fix a basis $\{e_i, f_j\mid 1\leq i \leq n-k, 1\leq j \leq k\}$ of $\CC^n$ on which $x$ acts as follows
Recall from \eqref{def:D'} that $D'_0 \equiv \CC^n$ has a basis $\{e_i, f_j\mid 1\leq i \leq n-k, 1\leq j \leq k\}$.
Let $x \in \textup{End}(\CC^n)$ be linear map determined by the assignments below:
\eq
f_k \mapsto f_{k-1} \mapsto \ldots \mapsto f_1 \mapsto 0,
\quad
e_{n-k} \mapsto e_{n-k-1} \mapsto \ldots \mapsto e_1 \mapsto 0.
\endeq
%----------------------------------------------------------------
For a cup $(i ,j) \in a$, we first note that $j$ must be greater than $i$ by the definition of a cup. Moreover, $j-i$ must be an odd number since the non-intersection condition forces that there are even vertices between $i$ and $j$.
Next, 
we denote its corresponding 
{\em cup relation} by
\eq\label{eq:cup_rel}
F_j=x^{-\frac{1}{2}(j-i+1)}F_{i-1},
\endeq
where $x^{-1}$ denotes the preimage of a space under the endomorphism $x$.
For a ray in $a$ connected to vertex $i$, we denote its corresponding {\em ray relation} by  
\eq\label{eq:ray_rel}
F_i=F_{i-1} \oplus \langle e_{\frac{1}{2}(i+\rho(i))}\rangle,
\endeq 
where
 $\rho(i)\in\mathbb Z_{>0}$ is the number of rays to the left of $i$ (including itself). We remark that \eqref{eq:ray_rel} is equivalent to 
$F_i = x^{-\frac{1}{2}(i-\rho(i))}(x^{n-k-\rho(i)} F_n )$.

\begin{proposition}\label{prop:known_results_about_irred_comp}
Let $F_\bullet \in \cB_{\ld}$, and let $a \in I_{\ld}$.
The flag $F_\bullet$ belongs to the irreducible component $K_a\subseteq \mathcal B_\ld$  if and only if
\[
\begin{split}
&\eqref{eq:cup_rel}\textup{ holds for any cup }(i,j) \in a,
\textup{ and}
\\
&\eqref{eq:ray_rel}\textup{ holds for any vertex }i \textup{ connected to a  ray in }a.
\end{split}
\]
\end{proposition}
\begin{proof}
An equivalent version of this proposition is first proved in \cite[Theorem~5.2]{Fun03}; while we present a simpler statement in  \cite[Proposition~7]{SW12}.
\end{proof}
%----------------------------------------------------------------
\ex\label{ex:Ka21}
Let $n=3, k=1$, and let
\[
a_1 = \cupa~,~ a_2 = \cupb~ \in B_{2,1}. 
\]
The cup relation for $(1,2)$ reads $F_2 = x\inv F_0 = x\inv \{0\} = \<e_1, f_1\>$; while the cup relation for $(2,3)$ suggests that $F_1 = x F_3 = x\<e_1, e_2, f_1\> = \<e_1\>$. 
Therefore, the two irreducible components in $\cB_{2,1}$ are  
\[
\begin{split}
&K_{a_1} 
= \{ F_\bullet \in \cB ~|~ x\inv F_0 = F_2, F_3 = F_2 \oplus \< e_2\>\}
= \{(0 \subset F_1 \subset \<e_1, f_1\> \subset \CC^3) ~|~ \dim F_1 = 1\} \simeq \CP^1,
\\
&K_{a_2} = \{F_\bullet \in \cB~|~ F_1 = F_0 \oplus \< e_1\>, x\inv F_1 = F_3\}
= \{(0 \subset \<e_1\> \subset F_2 \subset \CC^3) ~|~ \dim F_2 = 2\}\simeq \CP^1.
\end{split}
\]
We remark that the ray conditions are actually redundant in either irreducible component since each of them is implied by the complementary cup relation as follows:
\begin{align}
%\textup{cup relation for }(2,3) \Leftrightarrow 
&F_3 = x\inv F_1
&\Rightarrow &\quad F_1 = \<e_1\> = F_0 \oplus \<e_1\>,
%\Leftrightarrow \textup{ray relation for }1,
\\
%\textup{cup relation for }(1,2) \Leftrightarrow 
&F_2 = x\inv F_0 = \<e_1,  f_1\> 
&\Rightarrow &\quad F_3 = F_2 \oplus \<e_2\>.
%\Leftrightarrow \textup{ray relation for }3.
\end{align}
\endex
%----------------------------------------------------------------
%It is probably known that in any irreducible component, the ray relations are implied if all cup relations are satisfied. 
%We provide a proof here since we cannot find one in the literature.
\begin{lemma}
Let $F_\bullet$ be a complete flag in $\<e_i, f_j ~|~ 1\leq i \leq n-k, 1\leq j \leq k\>$, and let $a = \{ (i_t, j_t) ~|~1\leq t \leq k \} \in I_{n-k,k}$.
Assume that \eqref{eq:cup_rel} holds for all $(i,j) = (i_t, j_t), 1\leq t \leq k$. Then
\[
\eqref{eq:ray_rel} \textup{ holds for any }i \not\in \{i_t, j_t ~|~ 1\leq t \leq k\}.
\]
\end{lemma}
\proof
We use an induction on $n$. It is trivial when $n\leq 2$; while the case $n = 3$ is proved in Example~\ref{ex:Ka21}.
Assume now $n \geq 4$. Note that the vertex $n$ is either connected to a ray or a cup.

If vertex $n$ is connected by a cup to vertex $m$, then $m = n +1 - 2c$ for some non-negative integer $c$ we call the size of this cup, and hence the cup relation reads
\eq
F_{m-1} = x^{\frac{1}{2}(n-m+1)}F_n
= \< e_1, \ldots, e_{n-k-c}, f_1, \ldots, f_{k-c}\>.
\endeq 
Therefore, the induction hypothesis holds for the subdiagram of $a$ on the first $n-2c$ vertices, and hence all the ray relations hold.

If vertex $n$ is connected to a ray. Let vertex $m$ be the rightmost vertex that is connected to a cup.
Since all $k$ cup relations hold, $F_m$ must contain the subspace $\<f_1, \ldots, f_k\>$, and hence by dimension reason $F_m = \<e_1, \ldots, e_{m-k}, f_1, \ldots, f_k\>$.
Now $F_{m+1}$ is a $m+1$-dimensional space such that   
\eq
x F_{m+1} \subseteq
 \<e_1, \ldots, e_{m-k}, f_1, \ldots, f_k\>
 \subset
F_{m+1} \subset
\<e_1, \ldots, e_{n-k}, f_1, \ldots, f_k\>.
\endeq
It must be the case that $F_{m+1} =  \<e_1, \ldots, e_{m-k+1}, f_1, \ldots, f_k\>$, i.e., the ray relation holds for $m+1$.
By a similar argument, one shows that the ray relation holds for vertices $m+2$, $m+3$,$\ldots$ and $n$.
Note that the ray relation for any vertex $i \not\in\{i_t, j_t\}$ with $i < m$ also follows from the induction hypothesis for the subdiagram of $a$ on the first $m$ vertices.
\endproof
%%----------------------------------------------------------------
%\ex\label{ex:Ka22}
%Let $n=4, k=2$. We fix a basis $\{e_1, e_2, f_1, f_2\}$ of $\CC^3$ so that $x$ is determined by
%$e_2 \mapsto e_1 \mapsto 0, f_2 \mapsto f_1 \mapsto 0$.
%Define
%\[
%a_{12} = \cupaa~,~ a_{13}= \cupab~.
%\]
%The irreducible components in $\cB_{2,2}$ are 
%\[
%\begin{split}
%K_{a_{13}} 
%&= \{(x, F_\bullet) ~|~ x\inv F_0 = F_2, x\inv F_2 = F_4, \dim F_i = i\}
%\\
%&= \{(x, (0 \subset F_1 \subset \<e_1, f_1\> \subset F_3 \subset \CC^4)) ~|~ \dim F_1 = 1, \dim F_3=3\},
%\\
%K_{a_{12}} 
%&= \{(x, F_\bullet) ~|~ x^{-2} F_0 = F_4, x\inv F_1 = F_3, \dim F_i = i\}
%\\
%&= \{(x, F_\bullet) ~|~ x\inv F_1 = F_3, \dim F_i = i\}.
%\end{split}
%\]
%\endex
%----------------------------------------------------------------
%=============================================
\subsection{Irreducible components via quiver representations} \label{sec:irred}
%=============================================
For each cup $a\in I_{\ld}$, our strategy to single out the irreducible component $K_a \subset \cB_\ld \subset \tcS_{\ld}$ requires the following ingredients:
\begin{enumerate}
\item construction of a subset $\fT^a \subset \fT^+_{\ld}$ such that ${p}_{\bar{1}}(\fT^a) \simeq K_a$.
\item construction of a subset $\Lambda^a \subset \Lambda^+_{\ld}$ so that $\Phi(\Lambda^a) \simeq \fT^a$, which implies
$\Phi_M (p_\ld(\Lambda^a)) \simeq {p}_{\bar{1}}(\fT^a) \simeq K_a$.
\end{enumerate}
In light of \eqref{def:tphi}, we have
\eq\label{def:tphi2}
\begin{tikzcd}
\Lambda^+_{\bar{1}}
\ar[r, "{p_{\bar{1}}}", twoheadrightarrow] 
& 
M_{\bar{1}}
\ar[r, "{\tphi}", "\sim"'] 
&
\tcS_{\bar{1}}
\\
\fT^a
\ar[r, twoheadrightarrow] \ar[u, hookrightarrow]
& 
p_{\bar{1}}(\fT^a)
\ar[r, "\sim"', "\textup{Prop. }\ref{prop:a1}"] \ar[u, hookrightarrow]
&
K_a.
\ar[u, hookrightarrow]
\\
\Lambda^a
\ar[r, "p_\ld", twoheadrightarrow]
\ar[u, "\sim"', "\textup{Prop. }\ref{prop:a2}"]
&
p_\ld(\Lambda^a)
\ar[u, "\Phi_M", "\sim"']
&
\end{tikzcd}
\endeq
Recall that a cup diagram is uniquely determined by the configuration of its cups; that is, once the placement of the cups has been decided, rays emanate from the rest of the nodes.
Hence, for our construction of $\fT^a$ and $\Lambda^a$, we use only the information about the cups.
For completeness we also give a characterization for the ray relation on the quiver representation side.

Given a cup diagram $a = \{(i_t, j_t)~|~1\leq t \leq k\} \in I_{\ld}$,
we denote the set of all vertices connected to the left (resp.\ right) endpoint of a cup in $a$ by
\eq
V_l^a = \{i_t \mid 1\leq t \leq k\},
\quad
V_r^a = \{j_t \mid 1\leq t \leq k\}.
\endeq
Define the endpoint-swapping map by
\eq
\sigma: V_l^a\to V_r^a,
\quad
i_t \mapsto j_t.
\endeq
Given $i\in V_l^a$, denote the ``size'' of the cup $(i, \sigma(i))$ by
\eq
\delta(i)=\frac{1}{2}(\sigma(i)-i+1). 
\endeq
For instance, a minimal cup connecting neighboring vertices has size $\delta(i) = \frac{(i+1) -i + 1}{2} = 1$, and a cup containing a single minimal cup nested inside has size $2$.      

For short we set, for $0\leq p<q \leq n$,
\eq\label{eq:to}
\tB_{q\to p}
=
\tB_{p}\tB_{p+1}\cdots \tB_{q-1}: \tV_q \to \tV_p
,
\quad
\tA_{p \to q}
=
\tA_{q-1}\tA_{q-2}\cdots \tA_{p}: \tV_p \to \tV_q.
\endeq
Now we define
\eq\label{defi:M_n^a}
\fT^a
=\{(\tA,\tB,\tGa,\tDe)\in \fT^+_{\ld}
\mid
\ker \tB_{i+\delta(i)-1 \to i-1} 
= \ker \tA_{\sigma(i)-\delta(i) \to \sigma(i)} 
\tforall i\in V_l^a
\}.
\endeq 
The kernel condition in \eqref{defi:M_n^a} can be visualized in Figure~\ref{fig:Ma} below:

\begin{figure}[ht!]
\caption{The paths in the kernel condition of \eqref{defi:M_n^a}.}
\label{fig:Ma}
\[ 
\xymatrix@C=25pt@R=9pt{
	&   
	& 
	&
	\ar@{=}[d]
{\tV_{\sigma(i)-\delta(i)}}  
	\ar@/^/[r]^{\tA_{\sigma(i)-\delta(i)}}
	& 
{\tV_{i+\delta(i)}}  
	\ar@/^/[r]^{\tA_{\sigma(i)-\delta(i)+1}}
	&
	\cdots
   \ar@/^/[r]^{\tA_{\sigma(i)-1}}
	& 
%\underset{=\mathbb{C}^1}
{\tV_{\sigma(i)}.} 
\\
{\tV_{i-1}}  
	&   
	\ar@/^/[l]^{\tB_{i-1}}
{\tV_{i}}  
	& 
	\ar@/^/[l]^{\tB_{i}}
	\cdots
	&
	\ar@/^/[l]^{\tB_{i+\delta(i)-2}}
{\tV_{i+\delta(i)-1}}  
	& 
	&
	& 
%\underset{=\mathbb{C}^1}
}
\] 
\end{figure}

Note that for a minimal cup connecting neighboring vertices $i$ and $i+1$ the relations in \eqref{defi:M_n^a} take the simple form 
\eq
\begin{split}
\ker \tDe_1 = \ker \tA_1 
&\quad\tif i =1;
\\
\ker \tB_{i-1} = \ker \tA_i 
&\quad\tif 2 \leq i \leq n-2;
\\
\ker \tB_{n-2} = \tV_{n-1} = \CC 
&\quad\tif i=n-1.
\end{split}
\endeq 
We now show that the cup relation $F_j = x^{-\frac{1}{2}(j-i+1)}F_{i-1}$ on the flag side translates to the kernel condition $\ker \tB_{i+\delta(i)-1 \to i-1} 
= \ker \tA_{j-\delta(i) \to j}$ regarding transversal quiver representations.
%-----------------------------------------------------------------
\begin{prop}\label{prop:a1}
For $a \in B_{\ld}$,
we have an equality
$
p_{\bar{1}}(\fT^a) =  \tphi\inv(K_a).
$
\end{prop}
%-----------------------------------------------------------------
\proof
Thanks to Proposition~\ref{prop:MS}, it suffices to show that, for any $i\in V_l^a$ and
$(\tA, \tB, \tGa, \tDe) \in \fT^a$,
the kernel condition 
\eq\label{eq:cup_condition_quiver_side_rewritten} 
\ker \tB_{i+\delta(i)-1 \to i-1} 
= \ker \tA_{\sigma(i)-\delta(i) \to \sigma(i)}
\endeq 
is equivalent to the Fung/Stroppel--Webster cup relation (see Proposition~\ref{prop:known_results_about_irred_comp}~(b)(i))
\begin{equation}\label{eq:equivalent_condition}
(\tDe_1\tGa_1)^{-\delta(i)}
\ker \tGa_{1\to i-1}
=\ker \tGa_{1 \to \sigma(i)}.
\end{equation}
Note that the left-hand side of \eqref{eq:equivalent_condition} can be rewritten as follows:
\begin{align*}
(\tDe_1\tGa_1)^{-\delta(i)}
(\ker \tGa_{1 \to i-1}) 
	&= \ker (\tGa_{1 \to i-1} (\tDe_1\tGa_1)^{\delta(i)}) 
	\\
	&= \ker (\tA_{1 \to i-1} (\tB_1\tA_1)^{\delta(i)}  \tGa_1)	
	\\
	&= \ker (\tB_{i+\delta(i)-1 \to i-1}\tGa_{1\to i+\delta(i)-1}),
\end{align*}
where the second equality follows from applying $\delta(i)$ times the admissibility condition $\tB_1\tA_1=\tGa_1\tDe_1$;
while the third equality follows from applying
the admissibility condition $\tA_{t-1}\tB_{t-1} = \tB_t\tA_t$ repeatedly from $t=2$ to $t=i+\delta(i)-1 = \sigma(i)-\delta(i)$. 
Therefore, the cup relation 
\eqref{eq:equivalent_condition}
is equivalent to another kernel condition below
\eq\label{eq:final_equivalent_condition}
\ker \tB_{i+\delta(i)-1 \to i-1}\tGa_{1\to i+\delta(i)-1}
=\ker \tGa_{1 \to \sigma(i)}.
\endeq
In particular, the kernels are equal for the two maps in Figure~\ref{fig:cupeq} given by dashed and solid arrows, respectively:

\begin{figure}[ht!]
\caption{The kernel condition that is equivalent to \eqref{eq:equivalent_condition}.}
\label{fig:cupeq}
\[ 
\xymatrix@C=25pt@R=25pt{
{\tD_{i-1}}  
	\ar@/^/[d]^{\tGa_1}
	\ar@/_/@{.>}[d]_{\tGa_1}
\\
{\tV_{1}}
	\ar@/^/[r]^{\tA_{1}}
	\ar@/_/@{.>}[r]_{\tA_{1}}  
&
{\tV_{2}}
	\ar@/^/[r]^{\tA_{2}}
	\ar@/_/@{.>}[r]_{\tA_{2}}  
&
\dots  
	\ar@/^/[r]
	\ar@/_/@{.>}[r]  
	&
	\ar@{=}[d]
{\tV_{\sigma(i)-\delta(i)}}  
	\ar@/^/[r]^{\tA_{\sigma(i)-\delta(i)}}
	& 
	\cdots
   \ar@/^/[r]^{\tA_{\sigma(i)-1}}
	& 
%\underset{=\mathbb{C}^1}
{\tV_{\sigma(i)}} 
\\
&
{\tV_{i-1}}  
	&
	\ar@/^/@{.>}[l]^{\tB_{i}}   
	\cdots
	&
	\ar@/^/@{.>}[l]^{\tB_{i+\delta(i)-2}}
{\tV_{i+\delta(i)-1}}  
	& 
	& 
%\underset{=\mathbb{C}^1}
}
\] 
\end{figure}
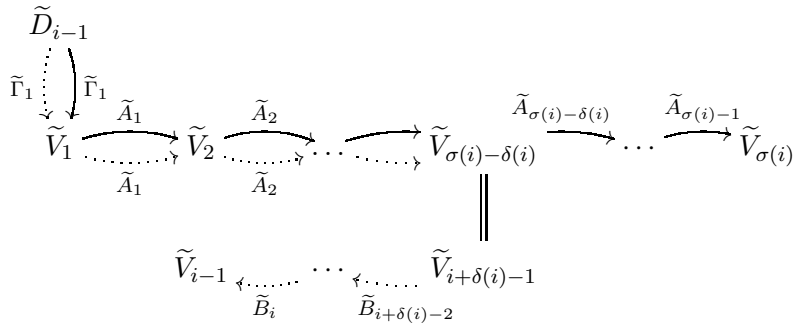

Note that~\eqref{eq:cup_condition_quiver_side_rewritten} evidently implies~\eqref{eq:final_equivalent_condition}. 
By the stability conditions on $\tV_1, \ldots, \tV_{i+\delta(i) -1}$ we see that the maps $\tGa_1, \tA_1, \ldots, \tA_{i+\delta(i)-2}$ are all surjective, and so is its composition $\tGa_{1\to i+\delta(i)-1}$. 
Thus, \eqref{eq:final_equivalent_condition} implies~\eqref{eq:cup_condition_quiver_side_rewritten}, and we are done.
\endproof
%-----------------------------------------------------------------

%===========================
\subsection{Irreducible components of Lagrangian subvarieties} \label{sec:IrrLagr}
%===========================
For any partition $\ld$, let $\cL_\ld = \tphi\inv_\ld(\cB_{\ld})$ be Lusztig's Lagrangian subvariety of $M_\ld$.
Let $\Lambda^\cB_{\ld}$ be the subset of $\Lambda^+_{\ld}$ so that
\eq
p_{\ld}(\Lambda^\cB_{\ld}) = \cL_\ld.
\endeq

The following characterization of the quiver representations corresponding to Springer fibers is well-known. See \cite[Lemma~5.9]{Na94} or \cite[Lemma~2.22]{Lu98}.

%--------------------------------------------------------------
\begin{prop}\label{prop:Spr}
If $x\in \cN$ is of Jordan type $\ld$ and $\lambda$ is an arbitrary partition of $n$, then
\[
\Lambda^\cB_{\ld} = \{(A,B,\Gamma,\Delta) \in \Lambda^+_{\ld}\mid\Delta = 0 \}.
\]
\end{prop}
%--------------------------------------------------------------

Now we can characterize the irreducible components using Lemma~\ref{lem:hell}.
Given a cup diagram $a\in I_{\ld}$, define 
\eq\label{defi:Springer_component_in_quiver}
\Lambda^a
=\{
(A,B,\Gamma,\Delta) \in \Lambda^\cB_{\ld}
\mid
\ker B_{i+\delta(i)-1 \to i-1} 
= \ker A_{\sigma(i)-\delta(i) \to \sigma(i)} 
\tforall i\in V_l^a
\},
\endeq
where the maps $A_{i\to j}, B_{j \to i}$ are defined similarly as their tilde versions in \eqref{eq:to}.
%------------------------------------------------
\begin{prop}\label{prop:hellA}
\label{prop:a2}
For $a \in I_{\ld}$,
we have an equality
$
\Phi(\Lambda^a) = \fT^a.
$
\end{prop}
%------------------------------------------------
\proof
Recall $\mu^+$ from \eqref{def:+}. We now define $\mu^{+(1)}:= \mu^+$ and $\mu^{+(i)} =(\mu^{+(i-1)})^+$ for all $i$.

Let $\Phi\inv(\tA, \tB, \tGa, \tDe) = (A,B, \Gamma, \Delta) \in \Lambda^a$.
From Proposition~\ref{prop:Spr} we see that $\Delta$ must be zero.
The proposition follows as long as we show that
$\ker \tB_{i+\delta(i)-1 \to i-1} 
= \ker \tA_{\sigma(i)-\delta(i) \to \sigma(i)}$ 
is equivalent to $\ker B_{i+\delta(i)-1 \to i-1} 
= \ker A_{\sigma(i)-\delta(i) \to \sigma(i)}$. 
For simplicity, let us use the shorthand $a = i-1 < b = i+\delta(i)-1 = \sigma(i) - \delta(i) < c = \sigma(i)$.

Using Lemma~\ref{lem:hell}, we obtain that,
for $a<b<c$, 
\eq
\tB_{b\to a}=
\small 
\begin{blockarray}{ *{8}{c} }
& V_{b} &   D'_{b}\\
\begin{block}{ c @{\quad} ( @{\,} *{7}{c} @{\,} ) }
V_a & {B}_{b\to a} & 0 \\
D'_{b}& 0 &   I   \\
D'_{\ld_a /\ld_b} & 0 &0\\
\end{block}
\end{blockarray}
\normalsize 
~~~~~,
\endeq
\eq
\tA_{b\to c}
=
\begin{blockarray}{ *{8}{c} }
& V_b & D''_b & \dots & D''_t & \dots & D''_{c-1} & D'_{\ld[c]^{+(c-b)}} \\
\begin{block}{ c @{\quad} ( @{\,} *{7}{c} @{\,} )}
V_{c}& {A}_{b\to c}&
A_{b+1\to c}\Gamma_{\to b+1}
&
\dots
&
A_{t+1\to c}\Gamma_{\to t+1}
&
\dots
&\Gamma_{\to c} & 0\\
D'_{c} & 0 &0 &\dots & 0 &
\dots & 0 &I   \\
\end{block}
\end{blockarray}
\normalsize
~~~~~,
\endeq
where $D''_t$ is the space (depending on fixed $b<c$) described below:
\eq\label{def:D''}
D''_t 
=
D'_{\ld[t]^{+(t-b)}/\ld[t+1]^{+(t-b+1)}}
=
\begin{cases}
\<e_{t-b+1}, f_{t-b+1}\> &\tif t \leq k, 
\\
\hspace{6mm} \<e_{t-b+1}\> &\tif k+1 \leq t \leq n-k, 
\\
\hspace{10mm} \{0\} &\textup{otherwise}.
\end{cases}
\endeq

Since $\tB_{b\to a}$ acts as an identity map on $D'_b$, its kernel must lie in $V_b$.
Moreover, $\ker \tB_{b\to a} = \ker B_{b\to a}$. 

Assume that $\ker \tB_{b \to a} 
= \ker \tA_{b \to c}$.
It follows that $\ker \tA_{b \to c} \subset V_b$.
In other words,
$D'_{\ld[b]/\ld[c-1]^{+(c-b)}}$ must not lie in the kernel, and hence $\ker A_{b\to c} =\ker \tA_{b\to c} = \ker \tB_{b \to a} = \ker B_{b\to a}$.

On the other hand, assuming $\ker B_{b \to a} 
= \ker A_{b \to c}$, we need to show that
$
D'_{\ld[b]/\ld[c-1]^{+(c-b)}}\not\in \ker
\tA_{b \to c}.
$
In other words, for $b \leq t \leq c-1$,
any composition of maps of the following form must be nonzero: 
\eq\label{eq:a23}
\begin{array}{cc}
\xymatrix@C=25pt@R=25pt{
&&
{D_{k}}  
	\ar@/^/[d]^{\Gamma_{k}}
\\
{V_{t+1}}
	\ar@/_/[r]_{A_{t+1 \to c}}  
&
{V_{c}}
	\ar@/_/[l]_{B_{c\to t+1}}  
&
	\ar@/_/[l]_{B_{k \to c}}	
{V_{k}}  
}
&
\xymatrix@C=25pt@R=25pt{
&&
{D_{n-k}}  
	\ar@/^/[d]^{\Gamma_{n-k}}
\\
{V_{t+1}}
	\ar@/_/[r]_{A_{t+1 \to c}}    
&
{V_{c}}
	\ar@/_/[l]_{B_{c\to t+1}}    
&
	\ar@/_/[l]_{B_{n-k \to c}}		
{V_{n-k}}  
}
\\
\\
\tif t+1 \leq k, 
&
\tif t+1 \neq n-k. 
\end{array}
\endeq

Note that the spaces $D''_t$ are nonzero only when $t \leq n-k$, and hence the maps $A_{t+1 \to c}\Gamma_{\to t+1}$ in \eqref{eq:a23} only exist when $c \leq n-k$.
Therefore, the proposition is proved for $c > n-k$, and we may assume now $c \leq n-k$.

We first prove that the composition of maps of the form in \eqref{eq:a23} must be nonzero for the base case $t=c-1$ by contradiction. 
Our strategy is to construct nonzero vectors $v_i \in V_i \cap \Im A_{1\to i}$ for $1\leq i \leq c$.
If this claim holds, then by admissibility conditions from $V_1$ to $V_i$, we have
\eq
B_{i-1}(v_i) = 
B_{i-1}A_{1 \to i}(v_{1}) =
A_{1\to i-1}B_1A_1(v_1) = 0,
\endeq
and hence there is a nonzero vector $v_b \in \ker B_{b\to a}$.
On the other hand, $v_b \not\in \ker A_{b\to c}$ since $A_{b\to c}(v_b) = v_c \neq 0$, a  contradiction.
Now we prove the claim. Suppose that $\Gamma_{n-k \to c} = 0$, then
\eq
\Gamma_{n-k \to i} = B_{c\to i}\Gamma_{n-k \to c}= 0, \quad \tforall i < c.
\endeq
By the stability condition on $V_1$, we have 
\eq
V_1 = \Im \Gamma_{k\to1} + \Im \Gamma_{n-k\to 1} = \Gamma_{k\to1}(f),
\endeq
 and hence the vector $\phi_i = \Gamma_{k\to i}(f) \in V_i$ are all nonzero for $i \leq k$.
Define $v_i= A_{1\to i}(\phi_1)$ for all $i$.
The stability condition on $V_2$ now reads
\eq
V_2 = \Im A_1 + \Im \Gamma_{k\to2} + \Im \Gamma_{n-k\to 2} = \<A_1(\phi_1)\> + \<\phi_2\>.
\endeq
Since $V_2$ is 2-dimensional, the vector $v_2 = A_1(\phi_1)$ must be nonzero. 
An easy induction shows that, for $2\leq i\leq k$, the vector $v_i$ is nonzero.
For $k +1\leq i \leq c$, the stability condition on $V_i$ is then
\eq
V_i = \Im A_{i-1} + \Im \Gamma_{n-k \to i} = A_{i-1}(V_{i-1}).
\endeq
Since both $\dim V_{i} = \dim V_{i-1} = k$, the map $A_{i-1}$ is of full rank, and hence $v_i \neq 0$ for $k+1\leq i \leq c$.
Therefore, we have seen that the assumption that $\Gamma_{n-k \to c} = 0$ leads to a contradiction, and hence $\Gamma_{n-k \to c} \neq 0$.
A similar argument shows that $\Gamma_{k \to c}\neq 0$.
The base case is proved.

Next, we are to show that the composition of maps of the form in \eqref{eq:a23} must be nonzero for $b\leq t < c-1$.
We write for short $h = c - t - 1$ to denote the size of the ``hook'' in the map $A_{t+1 \to c}\Gamma_{n-k \to t+1}$.
For example, as shown in the figure below, the maps $\Gamma_{\to c}$ have hook size 0, the maps $A_{c-1}\Gamma_{\to c-1}$ have hook size 1, and so on:
\[ 
\begin{array}{ccc}
\xymatrix@C=25pt@R=25pt{
&
{D_{j}}  
	\ar@/_/[d]_{\Gamma_{j}}
\\
{V_{c}}
&
	\ar@/_/[l]	
{V_{j}}  
}
&
\xymatrix@C=25pt@R=25pt{
&&
{D_{j}}  
	\ar@/_/[d]_{\Gamma_{j}}
\\
{V_{c-1}}
	\ar@/_/[r]_{A_{c-1}}
&
{V_{c}}
&
	\ar@/_/[ll]_{B_{j\to c-1}}
{V_{j}}  
}
&
\xymatrix@C=25pt@R=25pt{
&&
{D_{j}}  
	\ar@/_/[d]_{\Gamma_{j}}
\\
{V_{c-2}}
	\ar@/_/[r]_{A_{c-1}A_{c-2}}  
&
{V_{c}}
&
	\ar@/_/[ll]_{B_{j\to c-2}}
{V_{j}.}  
}
\\
\\
h=0
&
h=1
&
h=2
\end{array}
\] 
Note that $h$ is strictly less than the size $c-b$ of the cup.
Our strategy is to construct nonzero vectors $v_{i} \in V_{i} \cap \Im A_{h+1\to i}$ for $h+1\leq i \leq c$.
If this claim holds, then by admissibility conditions from $V_{h+1}$ to $V_i$ and an induction on $i$, we have
\eq
B_{i \to i-h-1}(v_i) = 
B_{i \to i-h-1} A_{i-1}( v_{i-1}) =  
A_{i-2} B_{i-1 \to i-h-2}( v_{i-1}) = 0.
\endeq
Note that the initial case holds since
$B_{h+1 \to 0}(v_{h+1}) = 0$.
Hence, there is a nonzero vector $v_b \in \ker B_{b\to b-h-1} \subset \ker B_{b\to a}$.
On the other hand, $v_b \not\in \ker A_{b\to c}$ since $A_{b\to c}(v_b) = v_c \neq 0$, a  contradiction.
We can now prove the claim.
Suppose first that $A_{t+1 \to c}\Gamma_{n-k \to t+1} = 0$.
By the admissibility conditions from $V_{t+2}$ to $V_{n-k+h-1}$, we have
\eq\label{eq:admt}
0 = A_{t+1 \to c}\Gamma_{n-k \to t+1} = 
B_{n-k+h\to c}A_{n-k \to n-k+h}\Gamma_{n-k},
\endeq
which can be visualized from the figure below by equating the two maps $D_{n-k}\to V_c$ represented by composing solid arrows and dashed arrows, respectively:
\[ 
\xymatrix@C=25pt@R=25pt{
&&
{D_{n-k}}  
	\ar@/^/@{.>}[d]^{\Gamma_{n-k}}
	\ar@/_/[d]_{\Gamma_{n-k}}
\\
{V_{t+1}}
	\ar@/_/[r]  
&
{V_{c}}
	\ar@/_/[l]  
&
	\ar@/^/@{.>}[l]
	\ar@/_/[l]	
{V_{n-k}}  
	\ar@/^/@{.>}[r]
	& 
	\ar@/^/@{.>}[l]
	{V_{n-k+h}.} 
}
\] 
For all $1 \leq i \leq t+1$, we will show now any map $D_{n-k} \to V_i$ with exactly a size $h$ ``hook'' is zero.
Precisely speaking, the admissibility conditions from $V_{i+1}$ to $V_{n-k+h-1}$ imply that
\eq
A_{i \to i+h}\Gamma_{n-k\to i}
=
B_{n-k+h \to i+h}A_{n-k \to n-k+h}\Gamma_{n-k},
\endeq 
which can be visualized as the figure below:
\[ 
\xymatrix@C=25pt@R=25pt{
&&
{D_{n-k}}  
	\ar@/^/@{.>}[d]^{\Gamma_{n-k}}
	\ar@/_/[d]_{\Gamma_{n-k}}
\\
V_i
	\ar@/_/[r]  
&
{V_{i+h}}
	\ar@/_/[l]  
&
	\ar@/^/@{.>}[l]
	\ar@/_/[l]	
{V_{n-k}}  
	\ar@/^/@{.>}[r]
& 
	\ar@/^/@{.>}[l]
	{V_{n-k+h}.}
}
\] 
It follows that $A_{i \to i+h}\Gamma_{n-k\to i} = 0$ since it is a composition of a zero map in \eqref{eq:admt}.
Since the hook size $h$ is less than the cup size $c-b$, which is less or equal to the total number $k$ of cups, we have
$h < k$ and so 
\eq
\dim V_h = h, 
\quad
\dim V_{h+1} = h+1.
\endeq
We claim that 
\eq\label{eq:claim1_t} 
\Im \Gamma_{k \to h+1} \neq 0.
\endeq
By the stability condition on $V_1$, we have 
\eq
V_1 = \Im \Gamma_{k\to 1} + \Im \Gamma_{n-k \to 1}.
\endeq
For the dimension reason, either $\Gamma_{k\to 1}$ or $\Gamma_{n-k\to 1}$ is nonzero. 
If $\Gamma_{k\to 1} \neq 0$ then $\Gamma_{k\to h+1} \neq 0$, and the claim follows.
If $\Gamma_{n-k\to 1} \neq 0$,
we define, for $1\leq i \leq l\leq n-k$, 
\eq
\epsilon_i = \epsilon^{(i)}_i = \Gamma_{n-k \to i}(e) \neq 0, 
\quad
\epsilon^{(i)}_l = A_{i\to l}(\epsilon_i).
\endeq
An easy induction shows that if $\epsilon^{(i)}_l = 0$ for some $1\leq i \leq l \leq h$ then $\Gamma_{k\to l} \neq 0$, which proves the claim. 
So we we now assume that $\epsilon^{(i)}_l \neq 0$ for all $1\leq i \leq l \leq h$.
Note that
\eq
\epsilon^{(1)}_{1+h}
 =
A_{1\to 1+h} \Gamma_{n-k \to 1}(e) = 0,
\endeq
and hence $\rk A_h = h-1$.
Now the  stability condition on $V_{h+1}$ implies that 
\eq
\dim V_{h+1} = \rk A_h + \rk \Gamma_{k\to h+1} + \rk \Gamma_{n-k\to h+1}, 
\endeq
and hence $\rk \Gamma_{k\to h+1} =1$.
The claim \eqref{eq:claim1_t} is proved.
Moreover, the vectors $\phi_i = \Gamma_{k\to i}(f) \in V_i$ are all nonzero for $h+1 \leq i \leq k$.
Define $v_i = A_{h+1 \to i}(\phi_{h+1})$ for all $i>h$.

The stability condition on $V_{h+2}$ now reads
\eq\label{eq:dimt}
V_{h+2} = \begin{cases}
\Im A_{h+1} + \Im \Gamma_{k\to h+2} + \Im \Gamma_{n-k\to h+2} 
&\tif h+2 \leq k, 
\\
\qquad \Im A_{h+1} + \Im \Gamma_{n-k\to h+2} 
&\tif h+2 > k.
\end{cases}
\endeq
In either case, a dimension argument similar to the one given in the base case $t=c-1$ shows that the vector $v_{h+2} = A_{h+1}(\phi_{h+1})$ must be nonzero since
\eq
\epsilon^{(i)}_{l}
 =
A_{i\to l} \Gamma_{n-k \to i}(e) = 0 
\quad
\textup{for}
\quad
l \geq i+h.
\endeq 
For $h+2\leq i\leq c$, a dimension argument using \eqref{eq:dimt} shows that
$v_i \neq 0$,
which leads to a contradiction, and hence $A_{t+1 \to c}\Gamma_{n-k \to t+1} \neq 0$.
A similar argument shows that $A_{t+1 \to c}\Gamma_{k \to t+1} \neq 0$.
The proposition is proved.
\endproof
%----------------------------------------------------------------------
\thm[Explicit quiver realization of irreducible components of Springer fibers]\label{thm:main1}
Recall $\Lambda^a$ from \eqref{defi:Springer_component_in_quiver}.
For any cup diagram $a \in I_{\ld}$, write $M_a = p_{\ld}(\Lambda^a)$. Then
\eq
M_a = \tphi_\ld\inv(K_a)
\endeq
is the irreducible component on the quiver side for the Springer fiber.
As a consequence, $\tphi_\ld\inv(\cB_{\ld}) = \bigcup_{a \in I_{\ld}} M_a$ is the decomposition into irreducible components.
\endthm
\proof
We have
\eq
\begin{aligned}
\tphi_\ld\inv(K_a)&={p}_{\bar{1}}(\fT^a)\quad &&\textup{by Proposition}~\ref{prop:a1}
\\
&=p_{\ld}(\Phi\inv(\fT^a))\quad &&\textup{by Proposition}~\ref{prop:Phi}(b)(c)
\\
&=p_{\ld}(\Lambda^a) \quad  &&\textup{by Proposition}~\ref{prop:a2}.
\end{aligned}
\endeq
\endproof
%----------------------------------------------------------------------
%=============================================
\subsection{The ray condition via quiver representations}
%=============================================
For completeness, in this section we characterize the ray condition $F_i = x^{-\frac{1}{2}(i-\rho(i))}(x^{n-k-\rho(i)} F_n ),$ on the quiver representation side.
%----------------------------------------------------------------------
\prop\label{prop:ray}
If $[A,B,\Gamma, 0]$ is sent to the irreducible component $K_a \subset M_{\ld}$ for a cup diagram $a \in B_\ld$.
Then the ray condition for vertex $i$ is equivalent to
\eq
\begin{cases}
\:\: B_iA_i = 0
&\tif c(i) \geq 1, 
\\
\Gamma_{n-k\to i} = 0
&\tif c(i) = 0,
\end{cases}
\endeq
where
$c(i) = \frac{i-\rho(i)}{2}$ is the total number of cups to the left of $i$.
\endprop
\proof
Write $\rho = \rho(i)$ and $c = c(i)$ for short.
By Theorem~\ref{thm:Fi}, the ray relation is equivalent to that
\eq\label{eq:rayblock}
\ker(A_{n-k-\rho+1 \to n} \Gamma_{\to n-k-\rho+1}~\dots~ A_{n-k \to n}\Gamma_{n-k})
=
\ker(A_{c+1 \to i}\Gamma_{\to c+1}~\dots~\Gamma_{\to i}),
\endeq
where the domains and codomains are dropped for simplicity.
Note that there are $\rho$ blocks on the left hand side of \eqref{eq:rayblock} and each block is a zero map; while there are $i -c = \rho + c \geq \rho$ blocks on the right hand side. 
Hence, the defining relations, by an elementary case-by-case analysis, are
\eq
\begin{cases}
A_{c+\rho \to i} \Gamma_{n-k\to c+\rho} = 0,
\quad
A_{c+\rho+1 \to i} \Gamma_{n-k\to c+\rho+1} \neq 0 
&\tif c \geq 1, 
\\
\qquad \qquad \qquad 
A_{c+\rho \to i} \Gamma_{n-k\to c+\rho} = 0 
&\tif c = 0.
\end{cases}
\endeq
Note that by definition, $i = \rho$ when $c = 0$.
We are done.
\endproof

%=============================================
\subsection{Irreducible components of Springer fibers of classical type}
\label{sec:ICSC}
%=============================================
For two-row Springer fibers of type C/D, it is known that they afford a realization in terms of fixed point sbvarieties of a two-row quiver varietiy $M_{(n-k,k)}$ under a  certain automorphism $\Theta$.
An explicit automorphism is introduced by Henderson--Licate in \cite{HL14} for the type A quivers.
A less explicit automorphism due to Li can be found in \cite{Li19}, which applies to all quivers associated with the symmetric pairs (or Satake diagrams).
It is unclear to us whether the two automorphisms match. 
However, the Maffei-Nakajima isomorphism $\tphi$ restricts to an isomorphism between the fixed-point subvariety $M_\ld^\Theta$ and a Slodowy variety of type C/D in either case (See \cite[(5.2), Thm.~5.3]{HL14}, \cite[Theorem~A]{Li19}). 

With our explicit quiver realization of irreducible components $M_a \subset M_\ld$, it is made possible to characterize irreducible components of type D Springer fibers via the fixed point subvarieties $M^\Theta_a \subset M^\Theta_\ld$.
An analogue of \eqref{defi:Springer_component_in_quiver}  is discovered in an unpublished manuscript \cite[\S5--7]{ILW19} by the authors.
The relations can then be translated via $\tphi_\ld$ as a type D generalization of the Stroppel--Webster result (Proposition~\ref{prop:known_results_about_irred_comp}) on the flag side.
We have done computations first on the quiver side to establish such a characterization of type D irreducible components.
We then realize that its counterpart on the flag side can be proved more efficiently, as long as we know these conditions which are derived by studying quiver representations that are fixed under $\Theta$. 
Therefore, we will provide details to the characterization of irreducible components of two-row Springer fibers of classical type in a sequel paper \cite{ILW20b} without referring to quiver varieties.

%=============================================
\bibliography{litlist-geom} \label{references}
\bibliographystyle{amsalpha}

\end{document}